%% file: main.tex
\documentclass[11pt]{amsart}
\usepackage{extarrows}
\usepackage[colorlinks, citecolor=blue]{hyperref}
\usepackage{amsthm,amsmath,amssymb}
\usepackage{mathrsfs}
\usepackage{bold-extra}
\usepackage{algpseudocode}
\usepackage[T1]{fontenc}
 \usepackage[ruled,vlined,linesnumbered]{algorithm2e}
\usepackage{tikz}
\usepackage{tikz-cd}

\newtheorem{theorem}{Theorem}[section]
\newtheorem{lemma}[theorem]{Lemma}
\newtheorem{mainthm}{Theorem}
\newtheorem{maincor}[mainthm]{Corollary}
\newtheorem{mainlem}[mainthm]{Lemma}

\theoremstyle{definition}
\newtheorem{definition}[theorem]{Definition}
\newtheorem{question}[theorem]{Question}
\newtheorem{example}[theorem]{Example}

\newtheorem{proposition}[theorem]{Proposition}
\newtheorem{corollary}[theorem]{Corollary}
\newtheorem{remark}[theorem]{Remark}
\newtheorem{conjecture}[theorem]{Conjecture}
\newtheorem{notation}[theorem]{Notation convention}
\theoremstyle{remark}

\usepackage[backend=biber, style=alphabetic,sorting=nyt,maxnames=99,maxalphanames=99]{biblatex}
\numberwithin{equation}{section}

\DeclareMathOperator{\chowpoly}{CH}

\begin{document}

\title{Inequalities for Chow Polynomials and Chern Numbers of Matroids}

\makeatletter
\def\shorttitle{Inequalities for Chow Polynomials and Chern Numbers of Matroids}
\makeatother

\author{Ronnie Cheng and Wangyang Lin}

\address{Department of Mathematics, Stanford University, USA}

\email{rtcheng@stanford.edu}

\address{School of Mathematical Sciences, Fudan University, Shanghai 200433, China}

\email{wylin23@m.fudan.edu.cn, wylin\_math@outlook.com}




\begin{abstract}
The Chow polynomial of a matroid is a fundamental invariant whose coefficients exhibit strong positivity properties, including $\gamma$-positivity. We interpret the normalized Chow coefficients as a probability distribution and establish new inequalities for its central moments. As consequences, we obtain bounds on the number of flags of flats and inequalities on the roots of the Chow polynomial.

We further relate these moment inequalities to algebraic geometry via the Hirzebruch $\chi_y$-genus. This yields new inequalities for matroidal Chern numbers. In particular, for any matroid of rank $d+1$, we prove that $c_1c_{d-1}\le c_d$, with equality if and only if $d=1$ or the simplification of the matroid is Boolean.
\end{abstract}

\maketitle

\tableofcontents

\input{1-introduction}

\input{2-preliminaries}

\input{3-Chow}

\input{4-c_top}

\input{5-appendix}
\printbibliography

\end{document}

%% file: 1-introduction.tex
\section{Introduction}

The Chow ring of a matroid was introduced (in the language of lattices) by Feichtner and Yuzvinsky \cite{FY_Chow}, who showed that it can be realized as the cohomology ring of a specific toric variety. This ring has attracted considerable attention in recent years, particularly following the proof of the Heron--Rota--Welsh conjecture on the log-concavity of the coefficients of the characteristic polynomial by Adiprasito, Huh, and Katz \cite{AHK18}. Their breakthrough relied on demonstrating that the Chow ring of a matroid exhibits the K\"ahler package, mirroring the formal properties enjoyed by the cohomology ring of a smooth projective variety. We refer to \cite{Ox} for background on matroid theory. All matroids in this paper are assumed to be \emph{loopless}.

For a matroid $\mathrm{M}$, we denote its Chow ring by $A^*(\mathrm{M})$. The \emph{Chow polynomial} of $\mathrm{M}$ is the Hilbert--Poincar\'e series of $A^*(\mathrm{M})$:
\[
\chowpoly_{\mathrm{M}}(x)\;:=\;\sum_{i\ge 0}\dim_{\mathbb{Q}}A^i(\mathrm{M})\,x^i,
\]
where $A^i(\mathrm{M})$ denotes the $i$-th graded piece. The polynomial $\chowpoly_{\mathrm{M}}(x)$ is palindromic, and by the Hard Lefschetz theorem \cite{AHK18}, its coefficients are unimodal.

Several positivity properties of Chow polynomials have recently been established. The property of $\gamma$-positivity for the Chow polynomial was first proved in \cite{Chowpoly} using the semi-small decomposition. Subsequently, an interpretation of the $\gamma$-coefficients using face enumerations was provided in \cite{ferroni2024chowfunctionspartiallyordered}. Building on this perspective, Stump \cite{Rlabel} proved $\gamma$-positivity by realizing Chow polynomials as evaluations of the Poincar\'e-extended $\mathbf{ab}$-index. 

Furthermore, the real-rootedness of the Chow polynomial, which implies both $\gamma$-positivity and coefficient log-concavity, has been conjectured independently by Ferroni--Schr\"oter \cite{FS_valuative_24} and by Huh--Stevens \cite{Stevens_realrooted}. Special cases have already been settled: Boolean matroids produce the Eulerian polynomials, and uniform matroids $\mathrm{U}_{d, d+1}$ are connected to the derangement polynomials (see \cite{Chowpoly, branden2025chowpolynomialsuniformmatroids}). Beyond matroids, Chow polynomial analogues and real-rootedness questions have been studied for bounded posets; see, for example, \cite{ferroni2024chowfunctionspartiallyordered, hoster2025chowpolynomialssimplicialposets, CFLchowpoly, branden2025chowpolynomialstotallynonnegative}.

The paper addresses a different question: how concentrated should the Chow coefficients of a matroid be? We establish a new family of inequalities for the coefficients of Chow polynomials of arbitrary matroids. Specifically, for any nonnegative integer $k$, we provide bounds on the ratio
\begin{equation}\label{eq:centralmomenttocoeff}
    \frac{\sum_{i=0}^d\left(i-\frac{d}{2}\right)^k\dim_{\mathbb{Q}}A^i(\mathrm{M})}{\sum_{i=0}^d\dim_{\mathbb{Q}}A^i(\mathrm{M})}.
\end{equation}

We formulate these bounds and their proofs in a probabilistic language. For a matroid $\mathrm{M}$ of rank $d+1$, we define a random variable $\mathcal{X}_{\mathrm{M}}$ by the probability mass function
\[
\Pr\!\big(\mathcal{X}_{\mathrm{M}}=p\big)\;=\;\frac{\dim_{\mathbb{Q}}A^p(\mathrm{M})}{\chowpoly_{\mathrm{M}}(1)}\qquad(p\ge0).
\]

For readers less familiar with probabilistic terminology, the $k$-th \emph{moment} of $\mathcal{X}_{\mathrm{M}}$ is the weighted sum $\mathbb{E}[\mathcal{X}_{\mathrm{M}}^k] = \sum_p p^k \Pr(\mathcal{X}_{\mathrm{M}}=p)$. Since the Chow polynomial is palindromic, the distribution has mean $\mu = d/2$. The $k$-th \emph{central moment} shifts this sum to the mean, giving
\[
\mathbb{E}\!\left[\big(\mathcal{X}_{\mathrm M} - d/2\big)^k\right]
\;=\;
\sum_p \left(p - \frac d2\right)^k \Pr(\mathcal{X}_{\mathrm{M}}=p).
\]
Under this interpretation, the ratio in \eqref{eq:centralmomenttocoeff} is exactly the $k$-th central moment. Bounding these central moments controls the dispersion and concentration of the Chow coefficients. These inequalities were first suggested by explicit computations of low-degree Chern numbers, specifically the inequality $c_1c_{d-1}\le c_d$.

To prepare for our main theorem, we introduce the concept of a \emph{central moment function sequence} (CMFS). This is a sequence of functions $\{f_0, \ldots, f_t\}$ from $\mathbb{Z}_{\geq 0}$ to $\mathbb{R}$ satisfying the convolution-like inequality
\[
f_k(a+b+2)\;\ge\;\sum_{i+j=k}\binom{k}{i}\,f_i(a)\,f_j(b),
\]
for all integers $a,b\ge0$ and all $0\le k\le t$. With this framework, our first main result shows that moment bounds for the Boolean matroid extend naturally to all matroids. 

\begin{mainthm}[Theorem~\ref{thm:generalineq}]\label{thm:A}
Let $\mathcal{X}_d$ denote the random variable associated with the Boolean matroid $\mathrm{U}_{d+1}=\mathrm{U}_{d+1,d+1}$. Suppose that a sequence of functions $\{f_0,\dots,f_t\}$ forms a CMFS such that for all $d\ge0$ and $0\le k\le t$,
\[
\mathbb{E}\!\Big[\big(\mathcal{X}_d-\tfrac{d}{2}\big)^k\Big]\;\le\; f_k(d).
\]
Then for any matroid $\mathrm{M}$ of rank $d+1$,
\[
\mathbb{E}\!\Big[\big(\mathcal{X}_{\mathrm{M}}-\tfrac{d}{2}\big)^t\Big]\;\le\; f_t(d).
\]
\end{mainthm}
We prove Theorem~\ref{thm:A} by induction, and the underlying recursion comes from the semi-small decomposition of matroids, which motivates the definition of CMFS. We can construct a CMFS inductively. In Section~\ref{sec:sharpness}, we propose an inductive procedure to find upper bounds for the $k$-th central moment. The bound in Corollary~\ref{cor:ineq_chow_coefficient} is obtained as the first nontrivial instance of this framework.

\begin{maincor}[Corollary~\ref{cor:ineq_chow_coefficient}]\label{cor:B}
Let $\mathrm{M}$ be a matroid of rank $d+1$. Then
\[
\mathbb{E}\!\Big[\big(\mathcal{X}_{\mathrm{M}}-\tfrac{d}{2}\big)^2\Big]\;\leq \;\frac{d+2}{12}.
\]
That is,
\[
\sum_{i=0}^d\left(i-\frac{d}{2}\right)^2\dim_{\mathbb{Q}}A^i(\mathrm{M})
    \;\le\; \frac{d + 2}{12}\sum_{i=0}^d \dim_{\mathbb{Q}}A^i(\mathrm{M}).
\]
Moreover, equality holds if and only if either $d = 1$ or the simplification of $\mathrm{M}$ is $\mathrm{U}_{d+1}$ with $d > 0$.
\end{maincor}

This inequality can be equivalently stated in terms of the Chow polynomial evaluated at $x=1$ as
\[
    \chowpoly_{\mathrm{M}}''(1) \;\le\; \frac{3d^2-5d+2}{12}\chowpoly_{\mathrm{M}}(1).
\]

Notably, this bound is not a consequence of the real-rootedness conjecture: even if one assumes that all matroid Chow polynomials are real-rooted, our result imposes an independent, nontrivial constraint on the coefficients. We also note that this specific inequality fails for augmented Chow rings, although the probabilistic methods developed in this paper can be adapted to yield analogous bounds with weaker constants in the augmented setting (see Remark~\ref{rem:augmented}).

Using similar methods, we can also produce upper bounds for the ratio of the $k$-th derivative \[
    \chowpoly_{\mathrm{M}}^{(k)}(1) \;\le\; g_k(d)\chowpoly_{\mathrm{M}}(1),
\] for explicitly computable functions $g_k(d)$ (Remark~\ref{rem:boundsfordifferential}). 

We also compare the Chow coefficient distribution with a normal distribution and obtain universal upper bounds for all moments. By setting
\[
f_k(d)\;=\; \mathbb{E}\!\Big[\mathcal{N}\big(0,\tfrac{d+2}{12}\big)^k\Big]
\;=\;
\begin{cases}
    0 &\text{if $k$ is odd,} \\
    \left(\frac{d+2}{12}\right)^{\frac{k}{2}}(k-1)!! &\text{if $k$ is even,}
\end{cases}
\]
where $\mathcal{N}(0,\tfrac{d+2}{12})$ is the normal distribution centered at $0$ with variance $\tfrac{d+2}{12}$, we obtain the following.

\begin{mainthm}[Theorem~\ref{thm:main-theorem}]
Let $\mathrm{M}$ be a matroid of rank $d+1$, and let $k\ge0$. Then
\[
\mathbb{E}\!\Big[\big(\mathcal{X}_{\mathrm{M}}-\tfrac{d}{2}\big)^k\Big]\;\le\;\mathbb{E}\!\Big[\mathcal{N}\big(0,\tfrac{d+2}{12}\big)^k\Big].
\]
\end{mainthm}

Since the Chow polynomial is $\gamma$-positive, there are also binomial bounds for the Chow coefficients (see Corollary~\ref{cor:binomial_bound}). For matroids of fixed rank $d$, the normal distribution gives sharper upper bounds for low moments, though these bounds weaken as $k$ grows. Conversely, the lower bounds for the central moments of $\mathcal{X}_\mathrm{M}$ are naive but perfectly sharp (see Corollary~\ref{cor:naiveboundformatroids}).

These coefficient inequalities carry several immediate structural consequences for matroids: bounds on the counts of flags of flats (Section~\ref{sec:reversed}), constraints on the $\gamma$-coefficients of the Chow polynomial (Section~\ref{sec:Rlabelings}), and necessary constraints on the roots of the Chow polynomial (Section~\ref{sec:real_rootedness}).

While these combinatorial corollaries are of independent interest, the primary motivation for bounding the central moments is geometric. The Chern--Schwartz--MacPherson (CSM) cycles of a matroid were first defined and studied in \cite{CSM_MRS}. Deeper structural results for the total CSM class via the notion of tautological classes $[Q_\mathrm{M}]$ and $[S_{\mathrm{M}}]$ of a matroid were developed in \cite{BEST}. When a matroid $\mathrm{M}$ on a ground set $E$ is realizable by a linear subspace $L \subseteq \Bbbk^{E}$, its CSM class has a geometric interpretation in terms of the Chern class of the log-tangent bundle $T{W}_L(-D)$. Here ${W}_L$ denotes the De\,Concini--Procesi wonderful compactification \cite{DCP} of the corresponding hyperplane arrangement, and $D$ is the boundary divisor. For arbitrary matroids, one can formally define the tangent $K$-class $T_{\mathrm{M}}$ \cite{Tbundle}. The Chern classes of this tangent bundle, denoted $c(\mathrm{M}) = c(T_{\mathrm{M}})$, are intricately linked to the coefficients of the Chow polynomial of $\mathrm{M}$ via the Hirzebruch $\chi_y$-genus, which translates the moment bounds into Chern number inequalities.

This relationship stems from the fact that when $\mathrm{M}$ is realized by $L\subseteq \mathbb{C}^E$, the Hodge numbers $h^{p,q}$ vanish for $p \neq q$. Consequently, we have
\[
\dim_{\mathbb{Q}}A^p(\mathrm{M})\;=\;\dim_\mathbb{C} H^{2p}({W}_L,\mathbb{C})\;=\;h^{p,p}({W}_L) \;=\;(-1)^p\chi\big({W}_L,\Omega^p_{{W}_L}\big),
\]
where the left-hand side is the Chow coefficient, and the right-hand side connects directly to Chern classes. We extend this relationship to non-realizable matroids using a valuativity argument, which we introduce in Section~\ref{sec:prelim-valuative}.

For a smooth complex projective variety (or more generally, a compact almost complex manifold) $X$ of complex dimension $d$, the Hirzebruch $\chi_y$-genus is defined as
\[
\chi_y(X)\;:=\;\sum_{p=0}^d \chi\big(X,\Omega^p_X\big)\,y^p.
\]
For $k \geq 0$ (and setting $0^0=1$), we define the evaluation polynomials:
\[
h_k(X)\;:=\; \sum_{p=0}^d (-1)^p\,\chi\big(X,\Omega^p_X\big) \left(p - \frac{d}{2}\right)^k.
\]

It is a classical result that the numbers $h_k(X)$ can be expressed entirely in terms of the Chern numbers of $X$. We record the first few terms of this expansion below.

\begin{mainlem}[{Lemma~\ref{lem:hirzeformula}, \cite[Lemma 2.1]{LiPing}}]
Let $X$ be a compact K\"{a}hler manifold of complex dimension $d$. If $k$ is odd, $h_k(X) = 0$. For even $k$ (provided that $d \geq k$), the expressions are:
\[
\begin{aligned}
 h_0(X)\;&=\;c_d(X),\\[3pt]
 h_2(X)\;&=\;\frac{d}{12}c_d(X)\;+\;\frac{1}{6}c_1c_{d-1}(X),\\[3pt]
 h_4(X)\;&=\;\frac{d(5d-2)}{240}c_d(X)\,+\,\frac{5d-2}{60}c_1c_{d-1}(X) \,+\, \frac{c_1^2+3c_2}{30}c_{d-2}(X)\, -\, \frac{c_1^3-3c_1c_2+3c_3}{30}c_{d-3}(X),\\
 \dots
\end{aligned}
\]
\end{mainlem}

By equating these geometric $h_k$ evaluations with our probabilistic central moments, coefficient inequalities for Chow polynomials immediately yield new inequalities for matroidal Chern numbers. For example, Corollary~\ref{cor:B} is turned into the following inequality.

\begin{mainthm}[Corollary~\ref{cor:ineq_of_chern_numbers}] \label{thm:intro_Chern_ineq}
Let $\mathrm{M}$ be a matroid of rank $d+1$. Then
\[
c_1 c_{d-1}(\mathrm{M}) - c_d(\mathrm{M}) \;\le\; 0,
\]
with equality precisely when $d=1$ or when the simplification of $\mathrm{M}$ is $\mathrm{U}_{d+1}$ with $d>0$.
\end{mainthm}

Notably, the inequality $c_1 c_{d-1}(\mathrm{M}) \leq c_d(\mathrm{M})$ is opposite in direction to the Fulton--Lazarsfeld inequality for nef vector bundles (see Remark~\ref{rem:Fulton-Lazarsfeld}). For realizable matroids, the Chern classes of the log-tangent bundle are computed from a short exact sequence involving the tautological subbundle $S_L|_{W_L}$ \cite[Theorem 8.8]{BEST}; moreover, the dual bundle $S_L^\vee$ is nef. Some of these underlying Chern class inequalities extend to arbitrary matroids \cite[Theorem 9.13]{BEST}. Thus, at a conceptual level, the Chern number inequalities for the tangent bundle and those for the log-tangent bundle point in opposite directions.

Classical Miyaoka--Yau-type inequalities give fundamental constraints on Chern numbers of smooth projective varieties, and they have played a central role in the study of positivity phenomena in algebraic geometry (see Remark \ref{rem:myineq}). In Section~\ref{subsec:alpha_truncation}, we produce a family of inequalities after multiplying by powers of $\alpha$, reflecting the effect of the principal truncation on the tangent bundle. We translate these Chern number inequalities into counts of flags of flats, proving several general Chern class inequalities for arbitrary matroids $\mathrm{M}$ (e.g., Corollary~\ref{cor:weakmapineq}). In particular, Corollary~\ref{cor:miyaoka-yau-alpha} provides an inequality stronger than the classical Miyaoka--Yau inequality for the divisor class $\alpha$.

\subsection*{Organization}
Section \ref{sec:pre} recalls background material.

In Section \ref{sec:moment}, we introduce the notion of \emph{central moment function sequences} (CMFS) and provide a systematic method to generate inequalities on the central moments. We provide bounds from normal distributions, establish the lower bounds and translate these results into counts of flags of flats, and then apply these inequalities to the $\gamma$-coefficients and the roots of the Chow polynomial.

In Section \ref{sec:Chern}, we translate the moment inequalities into relations among matroidal Chern numbers via the $\chi_y$-genus. As an application, we deduce the Chern-number inequality $c_1c_{d-1}\le c_d$ for matroids of rank at least two. Finally, we analyze the geometric effect of intersecting with the divisor class $\alpha$, proving further inequalities regarding principal truncations and Chern classes.

\subsection*{Acknowledgments}
The authors are grateful to Ping Li for inspiring the research on Chern number inequalities of matroids. We also thank Matt Larson for his helpful comments and suggestions regarding the exposition of this paper.

%% file: 2-preliminaries.tex
\section{Preliminaries}
\label{sec:pre}

\subsection{The wonderful variety}

A matroid $(\mathrm{M},E)$ of rank $d+1$ is \emph{representable} over the field $\Bbbk$ if there exists a linear subspace $L\subseteq \Bbbk^E$ such that the set of bases of $\mathrm{M}$ is exactly the collection $\{B\in\binom{E}{d+1}:L\cap\bigcap_{i\in B}H_i=\{0\}\}$, where $H_i$ is the $i$-th coordinate hyperplane and the binomial notation $\binom{E}{d+1}$ is the collection of all $(d+1)$-element sets of $E$. In this case, $L$ is called a \emph{representation} (or \emph{realization}) of $\mathrm{M}$. The condition that $\mathrm{M}$ is loopless means $L$ is not contained in any coordinate hyperplane.

We denote by ${W}_L$ the wonderful compactification of the complement $L\setminus \bigcup_{i\in E} (H_i\cap L)$ by iteratively blowing up the strict transforms of $\{L\cap \bigcap_{i\in F} H_i:F \text{ is a proper flat of }\mathrm{M}\}$ in increasing order of dimension (see \cite{DCP} for more details). We call ${W}_L$ the \emph{wonderful variety associated with} $\mathrm{M}$ via the realization $L\subseteq \Bbbk^E$. It is a smooth projective variety of dimension $d$. If $\mathrm{M}$ is Boolean and $L=\Bbbk^E$, the wonderful variety ${W}_L$ is exactly the permutahedral variety, which we will denote as $X_E$.

\subsection{Chow ring and \texorpdfstring{$K$}{K}-ring of a matroid}
Every matroid $\mathrm{M}$ on the ground set $E$ has an associated Bergman fan $\Sigma_\mathrm{M}$ \cite{Bergmanfan} and the corresponding toric variety $X_{\Sigma_\mathrm{M}}$, which we will denote by $X_\mathrm{M}$. This $X_\mathrm{M}$ sits naturally inside $X_E$. 

We define $A^*(\mathrm{M}):=A^*(X_\mathrm{M})$ to be the Chow ring of $X_\mathrm{M}$. When $\mathrm{M}$ is realizable by $L\subseteq \Bbbk^E$, we have
\[
W_L \;\hookrightarrow\; X_\mathrm{M} \;\hookrightarrow\; X_E,
\] and the inclusion ${W}_L\hookrightarrow {X}_\mathrm{M}$ induces an isomorphism $A^*(\mathrm{M})=A^*({X}_\mathrm{M})\cong A^*({W}_L)$. (See \cite[Remark 2.13]{BHM}.)

The Chow ring $A^*(\mathrm{M})$ admits a combinatorial presentation \cite[Section 5.3]{AHK18}:

\[
A^*(\mathrm{M}) = \mathbb{Z}[x_F \mid F \text{ a nonempty proper flat of } \mathrm{M}] / (I + J),
\]

where:
\begin{itemize}
    \item $I$ is the ideal generated by $x_F x_G$ for incomparable flats $F$ and $G$,
    \item $J$ is the ideal generated by $\sum_{F \ni i} x_F - \sum_{F \ni j} x_F$ for each pair of elements $i, j \in E$.
\end{itemize}

In \cite{LLPP}, the authors defined the $K$-ring of a matroid $\mathrm{M}$ by $K(\mathrm{M}):=K({X}_{\mathrm{M}})$. Similarly, when $\mathrm{M}$ is realizable via $L\subseteq \Bbbk^E$, the inclusion ${W}_L\hookrightarrow {X}_\mathrm{M}$ induces an isomorphism $K(\mathrm{M})\;\cong\; K(W_L)$.

\subsection{Tangent class of a matroid}

In \cite[Section 3]{Tbundle}, the tangent class $T_\mathrm{M} \in K(\mathrm{M})$ for a general matroid $\mathrm{M}$ was introduced and several properties were established. In particular, when $\mathrm{M}$ is realizable via $L\subseteq \Bbbk^E$, the class $T_\mathrm{M}$ coincides with the class of the tangent bundle $T_{W_L}$ of the wonderful compactification $W_L$.

\begin{definition}[{\cite[Definition 3.2]{Tbundle}}] 
For a matroid $\mathrm{M}$ of rank $d+1$ and an integer $k > 0$, we define $$Z_{k}(\mathrm{M})\; =\; \sum_{F \text{ is a rank } k \text{ flat }} x_F \in A^1(\mathrm{M}).$$
\end{definition}

Note that $Z_{k}(\mathrm{M}) = S_{d+1-k, \mathrm{M}}$ in the definition of the original paper; we will write $Z_k$ when there is no confusion. 

The total Chern class of the tangent class $T_{\mathrm{M}}$ is given by the following theorem.
\begin{theorem}[{\cite[Theorem~3.4]{Tbundle}}] \label{thm:Tangentbundle}
For a matroid $\mathrm{M}$ of rank $d+1$, the total Chern class of its tangent class $T_\mathrm{M}$ is given by (cf. Definition \ref{def:alphaandbeta})
\begin{equation}
    \label{eq:Chern_class_of_tangent}
    c(T_\mathrm{M})\; =\; \left(\prod_{i=1}^{d}(1+Z_{i})\right)\cdot\left(\prod_{i=0}^{d}(1+\alpha-\sum_{j=1}^i Z_{d+1-j})\right).
\end{equation}
\end{theorem}

We denote by $c(\mathrm{M})$ the total Chern class of $T_\mathrm{M}$. For a partition $\lambda$ of $n$ (that is, $\lambda = (\lambda_1, \ldots, \lambda_k)$ with $\sum_{i=1}^k \lambda_i = n$), we define the partition Chern class $c_\lambda(\mathrm{M}) = \prod_{i=1}^k c_{\lambda_i}(\mathrm{M})$. We will simply write $c_{\lambda}$ when the underlying matroid is clear from context.

\subsection{Valuativeness of flag-counting functions and Chern classes} \label{sec:prelim-valuative}
When a matroid is realizable, especially over $\mathbb{C}$, geometric and combinatorial properties are easier to verify. The notion of valuativeness gives us a way to extend these properties to non-realizable matroids.

A function $f$ from the class of matroids over $E$ to an abelian group is called \emph{valuative} if for any matroids $\mathrm{M}_1,\ldots,\mathrm{M}_k$ and integers $a_1,\ldots,a_k$ such that $\sum a_i\mathbf{1}_{P(\mathrm{M}_i)}=0$, we have $\sum a_if(\mathrm{M}_i)=0$, where $P(\mathrm{M})$ denotes the base polytope of $\mathrm{M}$ and $\mathbf{1}_{P(\mathrm{M})}$ is the indicator function of $P(\mathrm{M})$. 
\begin{proposition}[{\cite[Corollary 7.9]{EHL_Stella} or \cite[Theorem 5.4]{DF_val}}]
    For an arbitrary matroid $\mathrm{M}$, the indicator function $\mathbf{1}_{P(\mathrm{M})}$ can be expressed as a linear combination of indicator functions of Schubert matroids of the same rank. Moreover, every Schubert matroid is realizable over any infinite field.
\end{proposition}
Thus, if a valuative function is zero on every realizable matroid over $\mathbb{C}$, then it is zero on all matroids. In the following, we introduce several valuative functions.
\begin{definition}[{\cite[Definition 3.11]{Tbundle}}] \label{def:NJflagsofflats}
Let $\mathrm{M}$ be a rank $d+1$ matroid on the ground set $E$. For a subset $I=\{i_1,\ldots,i_k\}\subseteq[d] := \{1, \ldots, d\}$ with $i_1 < \cdots <i_k$, we define $N_I(\mathrm{M})$ as the number of flags of flats $\mathcal{G}:G_1\subsetneq G_2\subsetneq\cdots\subsetneq G_k$ such that $\operatorname{rk}(G_j)=i_j$ for all $j$.
\end{definition}

\begin{proposition} [{\cite[Theorem 6.2]{FS_valuative_24}}]
Fix the ground set $E$ and the number $d$. For a subset $I \subseteq [d]$, the map 
\[
\widehat{N}_{I}\colon \{\text{Matroids on } E \text { of rank $d+1$}\} \;\longrightarrow\; \mathbb{Z}\,,\qquad \mathrm{M}\; \longmapsto \;N_I(\mathrm{M})
\]
is valuative.
\end{proposition}

\begin{definition} \label{def:alphaandbeta}
Let $\mathrm{M}$ be a matroid. Define $\alpha_\mathrm{M} = \alpha = \alpha_i = \sum_{F \ni i} x_F \in A^1(\mathrm{M})$ and $\beta_\mathrm{M} = \beta = \beta_i = \sum_{F \not\ni i} x_F \in A^1(\mathrm{M})$. We write $\alpha$ and $\beta$ if the underlying matroid $\mathrm{M}$ is obvious. Due to the presentation of the Chow ring, $\alpha$ and $\beta$ are independent of $i$ (immediate for $\alpha$, and $\beta = (\sum_F x_F) -\alpha$).
\end{definition}

\begin{proposition} [{\cite[Proposition 3.7]{Tbundle}}] \label{prop:valuativefunctions}
Fix the ground set $E$ and the number $d$. For any degree $d$ homogeneous polynomial $f(x, y_1, \ldots, y_d)$, the map 
\[
\Phi_{f}\colon \{\text{Matroids on } E \text { of rank $d+1$}\} \;\longrightarrow\; \mathbb{Z}\,,\quad \mathrm{M}\; \longmapsto \;\deg_\mathrm{M}\left(f\left(\alpha_{\mathrm{M}}, Z_1(\mathrm{M}), \ldots, Z_d(\mathrm{M})\right)\right)
\]
can be written as a linear combination of $N_I$, and thus is valuative. 
\end{proposition}

By the formula of the total Chern class \eqref{eq:Chern_class_of_tangent}, the intersection $\alpha^a\beta^{b}c_{\lambda}(\mathrm{M})$ is a degree $d$ homogeneous polynomial of $\alpha,Z_1,\ldots,Z_d$ when $\lambda$ is a partition of $d-a-b$, and the intersection number is valuative on the set of rank $d+1$ matroids.

\subsection{Chow polynomial and its semi-small decomposition}

Semi-small decompositions of the Chow ring of a matroid were established in  \cite{BHM}, and the decomposition can be translated into a recursion of the Chow polynomial $\chowpoly_{\mathrm{M}}(x)$. 

\begin{theorem}[{\cite[Corollary 3.24]{Chowpoly}}] \label{cor:FMSV-3.24}
Let $\mathrm{M}$ be a matroid and let $i\in E$ be an element that is \emph{not} a coloop. Define $\underline{S}_i$ as the family of flats of $\mathrm{M}$ given by \begin{align*}
\underline{S}_i &= \{\,F \mid F \text{ is a flat},\ \varnothing\subsetneq F \subsetneq E\setminus\{i\},\ \text{and }F\cup\{i\}\text{ is a flat}\,\}.
\end{align*}

Then the Chow polynomial satisfies the identity
  \begin{equation}\label{eq:FMSV-3.24-chow}
    \chowpoly_{\mathrm{M}}(x)
    \;=\;
    \chowpoly_{\mathrm{M}\setminus\{i\}}(x)
    \;+\; x\sum_{F\in \underline{S}_i} \chowpoly_{\mathrm{M}/(F\cup\{i\})}(x)\cdot \chowpoly_{\mathrm{M}|F}(x).
  \end{equation} 
\end{theorem}

%% file: 3-Chow.tex
\section{Moment inequalities for Chow coefficients}
\label{sec:moment}
For a matroid $\mathrm{M}$ of rank $d+1$, we write
\[
  \chowpoly_{\mathrm{M}}(x)\;=\;\sum_{k=0}^d a_k(\mathrm{M})\,x^k,
  \qquad a_k(\mathrm{M})\;=\;\dim_{\mathbb{Q}} A^k(\mathrm{M}).
\]

\begin{definition} \label{def:rvmatroid}
For any matroid $\mathrm{M}$ of rank $d+1$, set $\mathcal{X}_\mathrm{M}$ as the random variable such that
\[
  \Pr(\mathcal{X}_\mathrm{M}=k) \;=\;\frac{a_k(\mathrm{M})}{\sum_{i=0}^d a_i(\mathrm{M})},\qquad k=0,\dots,d.
\]   
\end{definition}

The probability generating function (PGF) of $\mathcal{X}_{\mathrm{M}}$ is therefore $\chowpoly_{\mathrm{M}}(x)/\chowpoly_{\mathrm{M}}(1)$. Let $\mathrm{U}_{d+1}$ be the Boolean matroid on $n:=d+1$ elements; we denote by $\mathcal{X}_d$ the random variable $\mathcal{X}_{\mathrm{U}_{d+1}}$.

In this section, we provide a systematic way to produce inequalities regarding the coefficients of the Chow polynomial. The results and their proofs can be expressed cleanly in terms of the moments of the random variable $\mathcal{X}_\mathrm{M}$. 

\begin{definition}[(Nice) central moment functions / random variables]
\label{def:central_sequence}
Let $t\in\{0,1,2,\dots,\infty\}$.
A sequence of functions $\{f_0(d),\dots,f_t(d)\}$ from $\mathbb{Z}_{\geq 0}$ to $\mathbb{R}$ (with $f_0\equiv1$) is called a \emph{central moment function sequence} if for every integer $a,b \geq 0$ and every $0\le k\le t$ we have
\[
f_k(a+b+2)
\; \ge\;
\sum_{i+j=k}\binom{k}{i}\, f_i(a)\, f_j(b).
\]

Suppose that we have \emph{independent} random variables $\mathcal{Y}_a,\mathcal{Y}_b$ with moments $\mathbb{E}[\mathcal{Y}_a^r]=f_r(a)$ and $\mathbb{E}[\mathcal{Y}_b^r]=f_r(b)$ for all $r\ge0$; then the conditions of being a central moment function sequence are equivalent to
\[
\mathbb{E}\!\big[\mathcal{Y}_{a+b+2}^k\big]
\;\ge\;
\mathbb{E}\!\Big[\big(\mathcal{Y}_a+\mathcal{Y}_b\big)^k\Big].
\]

We say the sequence is \emph{nice} if equality holds for all nonnegative integers $a,b,k$, i.e.,
\[
f_k(a+b+2)
\; =\;
\sum_{i+j=k}\binom{k}{i}\, f_i(a)\, f_j(b).
\]

An infinite sequence of random variables $\{\mathcal{Y}_d\}_{d\geq 0}$ is called a \emph{(nice) central moment random variable sequence} if the sequence of functions $\{f_k\}_{k \geq 0}$ is a (nice) central moment function sequence, where $f_k(d) = \mathbb{E}[\mathcal{Y}_d^k]$.

We will abbreviate the notion of central moment function sequence to CMFS, and the notion of central moment random variable sequence to CMRVS.
\end{definition}

In particular, the sequence of random variables $\{\mathcal{Y}_d\}$ is a nice CMRVS if it satisfies $$\mathcal{Y}_{a+b+2} \; = \; \mathcal{Y}_a \; + \; \mathcal{Y}_b.$$

The following theorem is the main technique by which we establish inequalities, and the above definitions and names are motivated by it.
\begin{theorem} \label{thm:generalineq}
Suppose that the functions $\{f_0, \ldots, f_t\}$ form a CMFS. If for all $d \geq 0$ and $0 \leq k \leq t$, $\mathbb{E}[(\mathcal{X}_d - \frac{d}{2})^k] \leq f_k(d)$, then for any matroid $\mathrm{M}$ of rank $d+1$, we have the inequality $$\mathbb{E}\!\Big[\big(\mathcal{X}_{\mathrm{M}}-\frac{d}{2}\big)^t\Big] \; \leq \; f_t(d).$$
\end{theorem}
\begin{proof}
We may assume $\mathrm{M}$ to be simple, as simplification would not affect its Chow polynomial $\chowpoly_{\mathrm{M}}$ and hence the random variable $\mathcal{X}_{\mathrm{M}}$. 

A prefix of a CMFS is still a CMFS. Applying induction on $t$, we may assume that we already know $$\mathbb{E}\!\Big[\big(\mathcal{X}_{\mathrm{M}}-\tfrac{d}{2}\big)^s\Big] \;\leq\; f_s(d), \quad s = 0, \ldots, t - 1.$$ 

We apply induction again on the number $n$ of elements of the ground set. For a matroid $\mathrm{M}$ of rank $d+1$, if $\mathrm{M}$ is Boolean, the inequality holds by the assumption about $\mathcal{X}_d$; otherwise, suppose $i$ is a non-coloop element. Recall \eqref{eq:FMSV-3.24-chow} that  
\begin{equation}
    \chowpoly_{\mathrm{M}}(x)
    \;=\;
    \chowpoly_{\mathrm{M}\setminus\{i\}}(x)
    \;+\; x\sum_{F\in \underline{S}_i} \chowpoly_{\mathrm{M}/(F\cup\{i\})}(x)\cdot \chowpoly_{\mathrm{M}|F}(x).
\end{equation} 
For a matroid $\mathrm{M}$, write $S(\mathrm M):=\sum_{k=0}^d a_k(\mathrm M)=\chowpoly_{\mathrm M}(1)$ for the sum of the coefficients. Therefore,
\begin{align*}
S(\mathrm M)\,\mathbb{E}\!\Big[\big(\mathcal{X}_{\mathrm{M}}-\tfrac{d}{2}\big)^t\Big]
&= S(\mathrm{M}\setminus\{i\})\,\mathbb{E}\!\Big[\big(\mathcal{X}_{\mathrm M\setminus\{i\}}-\tfrac{d}{2}\big)^t\Big] \\
&+ \sum_{F\in\underline S_i}\; 
  \sum_{u,v\ge0} (1+u+v-\tfrac{d}{2})^t \, a_u\big(\mathrm{M}/(F\cup\{i\})\big)\, a_v\big(\mathrm{M}|F\big).
\end{align*}
For a fixed $F\in\underline S_i$, let
\[
S_1\;:=\;S\big(\mathrm{M}/(F\cup\{i\})\big), \qquad
S_2\;:=\;S\big(\mathrm{M}|F\big), 
\]
and let $\mathcal X^{(1)} := \mathcal X_{\mathrm{M}/(F\cup\{i\})}$, $\mathcal X^{(2)} := \mathcal X_{\mathrm{M}|F}$ denote the corresponding random variables. The double sum over $u,v$ equals
$$S_1\,S_2\,\mathbb{E}\!\Big[\big(\mathcal{X}^{(1)}-\tfrac{a}{2}+\mathcal{X}^{(2)}-\tfrac{b}{2}\big)^{t}\Big],$$
where $a=\deg\chowpoly_{\mathrm{M}/(F\cup\{i\})}$, $b=\deg\chowpoly_{\mathrm{M}|F}$, and $a+b=d-2$. By the induction hypothesis on the number of elements, both matroids $\mathrm{M}/(F\cup\{i\})$ and $\mathrm{M}|F$ have strictly fewer elements than $\mathrm M$, hence for every $0\le s\le t$ their $s$-th central moments are bounded by the corresponding $f_s$ evaluated at their degrees.  Expanding
\begin{align*}
 \mathbb{E}\!\Big[\big(\mathcal X^{(1)} - \tfrac{a}{2}+\mathcal X^{(2)}- \tfrac{b}{2}\big)^t\Big]
\;&=\;\sum_{i+j=t}\binom{t}{i}\,
\mathbb{E}\!\Big[\big(\mathcal X^{(1)}- \tfrac{a}{2}\big)^i\Big]\,\mathbb{E}\!\Big[\big(\mathcal X^{(2)}- \tfrac{b}{2}\big)^j\Big]\\
&\le \;\sum_{i+j=t}\binom{t}{i}\, f_i(a)\, f_j(b) \leq f_t(a+b+2) \;= \;f_t(d),
\end{align*}
where we bound each $\mathbb{E}[(\mathcal X^{(1)}- \frac{a}{2})^i]\le f_i(a)$, $\mathbb{E}[(\mathcal X^{(2)}- \frac{b}{2})^j]\le f_j(b)$, and the second inequality follows from the defining property of a CMFS. Hence, for every $F\in\underline S_i$,
\[
\sum_{u,v\ge0} \big(1+u+v-\tfrac{d}{2}\big)^t \, a_u\big(\mathrm{M}/(F\cup\{i\})\big)\, a_v\big(\mathrm{M}|F\big)
\;\le\; S_1S_2\, f_t(d).
\]

Combining the above bounds over all $F$ and the deletion term yields
\[
S(\mathrm M)\,\mathbb{E}\!\Big[\big(\mathcal{X}_{\mathrm{M}}-\tfrac{d}{2}\big)^t\Big] \;\le\; S(\mathrm{M}\setminus\{i\})\,f_t(d) \;+\; \sum_{F\in\underline S_i} S_1S_2\, f_t(d) \;=\; S(\mathrm{M})f_t(d).
\]
Dividing by $S(\mathrm M)$ completes the induction.
\end{proof}

\begin{corollary}\label{cor:ineq_moments}
For any matroid $\mathrm{M}$ of rank $d+1$ we have
\[
\mathbb{E}\!\Big[\big(\mathcal{X}_{\mathrm{M}}-\tfrac{d}{2}\big)^2\Big]\;\le\; \frac{d+2}{12}.
\]
\end{corollary}
\begin{proof}
One can verify that the sequence of functions $\{f_0, f_1, f_2\}$ defined by $f_0(d) = 1$, $f_1(d) = 0$, and $f_2(d) = (d+2)/12$ forms a CMFS. To apply Theorem \ref{thm:generalineq}, we need to check the inequality for Boolean matroids.

Only the case $k=2$ is nontrivial. In fact, it is standard that when $d\geq 1$ we have $\mathbb{E}[(\mathcal{X}_d - \frac{d}{2})^2] = \frac{d+2}{12}$ (this number is $0$ when $d=0$). We will provide a way to compute $\mathbb{E}[(\mathcal{X}_d - \frac{d}{2})^k]$ when $d \geq k-1$ in Corollary~\ref{cor:equalwhendgeqk-1}.
\end{proof}

We can find the condition for equality to hold by tracking through the induction hypothesis in Theorem \ref{thm:generalineq}. The following lemma provides a way to analyze the equality cases.
\begin{lemma} \label{lem:flatinSi}
Let $\mathrm{M}$ be a simple matroid of rank at least $3$ on ground set $E$, and let $i \in E$ be a non-coloop such that the deletion $\mathrm{M} \setminus \{i\}$ is a Boolean matroid. Then there exists a rank-$1$ flat $F \in \underline{S}_i$.
\end{lemma}

\begin{proof}
Since $\mathrm{M} \setminus \{i\}$ is Boolean, $B = E \setminus \{i\}$ is a basis of $\mathrm{M}$ with $|B| = \operatorname{rk}(\mathrm{M}) \ge 3$. 
Let $C$ be the fundamental circuit of $i$ with respect to $B$, so $C \subseteq B \cup \{i\} = E$ and $i \in C$. Because $\mathrm{M}$ is simple, $|C| > 2$. 

If $|C| = 3$, then $C = \{i, a, b\}$ and $E = C \cup (E \setminus C)$; choosing any $j \in E \setminus C$ gives $\operatorname{cl}(\{i, j\}) = \{i, j\}$, so $F = \{j\}$ works. 
If $|C| > 3$, then for any $j \in C \setminus \{i\}$ the pair $\{i, j\}$ is not contained in a smaller circuit, hence $\operatorname{cl}(\{i, j\}) = \{i, j\}$. 
Thus, in either case, there exists $j \ne i$ with $F = \{j\}$ satisfying that $F$ is a rank-$1$ flat and $F \cup \{i\}$ is a rank-$2$ flat.
\end{proof}
\begin{corollary} \label{cor:ineq_chow_coefficient}
Let $\mathrm{M}$ be a matroid of rank $d+1$. Then $\mathbb{E}[(\mathcal{X}_{\mathrm{M}}-\tfrac{d}{2})^2]\leq \frac{d+2}{12}$. That is,
  \begin{equation}\label{eq:ineq_chow_coefficient_repeat}
    \sum_{i=0}^d\left(i-\frac{d}{2}\right)^2a_i(\mathrm{M})
    \;\le\;
    \frac{d + 2}{12}\sum_{i=0}^d a_i(\mathrm{M}).
  \end{equation}
Moreover, equality holds if and only if either $d = 1$ or the simplification of $\mathrm{M}$ equals $\mathrm{U}_{d+1}$ with $d > 0$. 
\end{corollary}
\begin{proof}
Without loss of generality, we assume that $\mathrm M$ is simple. To deduce the characterization of the equality case, note that \eqref{eq:ineq_chow_coefficient_repeat} does not reach equality when $d = 0$. Therefore, it suffices to prove that at least one rank 1 matroid would appear in the summation over $\underline{S}_i$ in the proof of Theorem~\ref{thm:generalineq}. 

For a simple non-Boolean matroid $\mathrm{M}$ of rank at least 3, iteratively deleting coloops must eventually reach a step where the deletion produces a Boolean matroid. When $\mathrm{M}\setminus\{i\}$ is Boolean, Lemma \ref{lem:flatinSi} guarantees the existence of a rank-$1$ flat $F\in\underline{S}_i$. This finishes the proof of Corollary \ref{cor:ineq_chow_coefficient}.
\end{proof}

\begin{corollary} \label{cor:generalineq}
Suppose that the sequence of random variables $\{\mathcal{Y}_d\}$ is a CMRVS. If for all $d, k \geq 0$, $\mathbb{E}[(\mathcal{X}_{d}-\frac{d}{2})^k] \leq \mathbb{E}[\mathcal{Y}_d^k]$, then for any matroid $\mathrm{M}$ of rank $d+1$, we have $$\mathbb{E}\!\Big[\big(\mathcal{X}_{\mathrm{M}}-\tfrac{d}{2}\big)^k\Big] \;\leq\; \mathbb{E}\!\big[\mathcal{Y}_d^k\big].$$
\end{corollary}

\begin{remark}
When $d\geq 0$, the polynomial $x^k$ is a nonnegative linear combination of $(x-\frac{d}{2})^t$ with $0\leq t \leq k$. Therefore, inequalities for the central moments $\mathbb{E}[(\mathcal{X}_{\mathrm{M}}-\frac{d}{2})^t] \leq \mathbb{E}[\mathcal{Y}_d^t]$ for all $t \leq k$ imply $\mathbb{E}[\mathcal{X}_\mathrm{M} ^t] \leq \mathbb{E}[(\mathcal{Y}_d+\frac{d}{2})^t]$ for all $t \leq k$.

The hypothesis on central moments is particularly convenient in practice. $\mathcal{X}_d$ is symmetric in the sense that for every $s$ we have $$ \Pr(\mathcal{X}_d = s)\;=\;\Pr(\mathcal{X}_d = d-s).$$ Therefore, $\mathbb{E}[(\mathcal{X}_d-\tfrac d2)^k]=0$ for every odd $k$, and only the even central moments enter the comparison. 
\end{remark}

A natural continuation is to define the ``reversed'' version of the inequality and obtain a lower bound for $\mathbb{E}[(\mathcal{X}_\mathrm M-\frac{d}{2})^k]$. However, the best numeric bound is the trivial bound. We will discuss it in Section~\ref{sec:reversed}.

\begin{remark} \label{rem:augmented}
Theorem~\ref{thm:generalineq} can be generalized to the augmented Chow ring using essentially the same proof. However, because the corresponding semi-small decomposition involves terms from the non-augmented Chow ring, applying the inductive step requires the bound to hold for the Boolean case of both the standard and the augmented Chow polynomials simultaneously.
\end{remark}

\subsection{Properties of Eulerian numbers}

Theorem \ref{thm:generalineq} tells us that we can extend the inequalities on $\mathcal{X}_d$ to $\mathcal{X}_\mathrm{M}$ for arbitrary matroids $\mathrm{M}$. In this subsection, we provide methods to compute the Boolean case.

Let $\mathrm{U}_{d+1}$ be the Boolean matroid on $n:=d+1$ elements. It is standard that the graded dimensions of $A^*(\mathrm{U}_{d+1})$ are given by the Eulerian numbers (See \cite[Section 2.3]{Chowpoly}):
\[
  a_k(\mathrm{U}_{d+1}) \;=\; A(n,k) \qquad (0\le k\le d),
\]
where $A(n,k)$ is the number of permutations in $S_n$ with exactly $k$ descents. Thus,
\[
  \chowpoly_{\mathrm{U}_{d+1}}(x)\;=\;\sum_{k=0}^{d} A(d+1,k)\,x^k
\]
is the (classical) Eulerian polynomial of order $d+1$ (the degree equals $d$). Therefore, the random variable $\mathcal{X}_d$ is distributed identically to the number of descents of a uniformly random permutation in $S_{d+1}$.

For integers $n, k \geq 0$, we let $S(n, k)$, the Stirling number of the second kind, be the number of ways to partition a set of 
$n$ labelled objects into $k$ nonempty unlabelled subsets.
\begin{theorem} \label{thm:Ekpoly}
If we fix an integer $k\ge0$, the moment $\mathbb{E}[\mathcal{X}_d^k]$ is a polynomial in $d$ of degree $k$ for all $d\ge k-1$. Denote this polynomial by $E_k(d)$, and define $\overline{\mathcal{X}_d}$ to be the formal random variable with moments $\mathbb{E}[\overline{\mathcal{X}_d}^k]=E_k(d)$ for all $k$ (i.e., $\overline{\mathcal{X}_d}$ is an algebraic object treated as a random variable whose moments are defined by the sequence $\{E_k(d)\}_{k \ge 0}$). Then
\[
\mathbb{E}\!\big[e^{z\overline{\mathcal{X}_d}}\big]
\;=\; \sum_{k\ge0}\frac{E_k(d)}{k!}z^k
\;=\; \left(\frac{e^z-1}{z}\right)^{d+2} e^{-z}.
\]
In particular, the random variables $\{\overline{\mathcal{X}_d} - \frac{d}{2}\}_{d \geq 0}$ form a nice CMRVS.   
\end{theorem}

\begin{proof}
Write $\mathcal{X}_d=\sum_{i=1}^d I_i$, where $I_i$ is the indicator of a descent at position $i$. Then
\[
\mathbb{E}\!\big[\mathcal{X}_d^k\big]
\;= \;\sum_{i_1,\ldots,i_k}
\Pr(I_{i_1}\;=\;\cdots=I_{i_k}=1).
\]
Suppose that the incidences $\{i_1,\dots,i_k\}$ form $m$ contiguous blocks of lengths $n_1,\dots,n_m$. The probability that all of them are descents is $\prod_{j=1}^m \frac{1}{(n_j+1)!}$, and the number of ways to place these blocks is $\binom{d-t+1}{m}$, where $t=n_1+\cdots+n_m$. Hence
\[
\mathbb{E}\!\big[\mathcal{X}_d^k\big]
\;=\; \sum_{\substack{n_1+\cdots+n_m=t\\ n_i>0,\, t\le k,\, t\le d}}
\left(t!\,S(k,t)\,\binom{d-t+1}{m}\prod_{j=1}^m\frac{1}{(n_j+1)!}\right).
\]

For $d\ge k-1$ the condition $t\le d$ can be removed, so the expression becomes a polynomial in $d$, denoted $E_k(d)$. We rewrite it as
\[
E_k(d)
\;=\; \sum_{t\le k} t!\,S(k,t) \,
   \sum_m \binom{d-t+1}{m} \,
   \sum_{\substack{n_1+\cdots+n_m=t\\ n_i>0}}\,
   \prod_{j=1}^m\frac{1}{(n_j+1)!}.
\]

The inner sums can be rewritten:
\[
\sum_{\substack{n_1+\cdots+n_m=t\\ n_i>0}}
\prod_{j=1}^m\frac{1}{(n_j+1)!}
\;=\; [x^t]\left(\frac{e^x-1-x}{x}\right)^m,
\]
and hence
\[
\sum_m \binom{d-t+1}{m} \left[x^t\right]\left(\frac{e^x-1-x}{x}\right)^m
=\; \left[x^t\right] \left(1+\frac{e^x-1-x}{x}\right)^{d-t+1}
=\; \left[x^t\right]\left(\frac{e^x-1}{x}\right)^{d-t+1}.
\]

Next,
\[
\left[x^t\right]\left(\frac{e^x-1}{x}\right)^{d-t+1}
= \;\operatorname{Res}_{x=0}\frac{(e^x-1)^{d-t+1}}{x^{d+2}}
\;=\; \left[u^t\right]\frac{u^{d+2}}{(\log(1+u))^{d+2}(1+u)},
\]
via the substitution $u=e^x-1$.

Thus,
\[
E_k(d)
\;=\; \sum_{t\le k} t!\,S(k,t)\,
   \left[x^t\right]\frac{x^{d+2}}{(\log(1+x))^{d+2}(1+x)}.
\]

Since the exponential generating function for Stirling numbers of the second kind satisfies
\[
\sum_{k\ge t} t!\,S(k,t)\frac{z^k}{k!}
\;=\; (e^z-1)^t,
\]
we obtain
\begin{align*}
\mathbb{E}\!\big[e^{z\overline{\mathcal{X}_d}}\big]
&\;=\; \sum_{k\ge0} \frac{E_k(d)}{k!}z^k \\
&\;=\; \sum_{t\ge0} \left[x^t\right]\frac{x^{d+2}}{(\log(1+x))^{d+2}(1+x)} (e^z-1)^t \\
&\;=\; \left(\frac{e^z-1}{z}\right)^{d+2} e^{-z}.
\end{align*}
Therefore, 
\begin{align*}
\mathbb{E}\!\left[e^{z(\overline{\mathcal{X}_d}-\frac d2)}\right]
&\;=\;\left(\frac{e^z-1}{z}\right)^{d+2} e^{-z}e^{\frac{-dz}{2}} \\
&\;=\;\Big(\frac{e^{z/2}-e^{-z/2}}{z}\Big)^{d+2}, 
\end{align*}
and $\mathbb{E}[e^{z(\overline{\mathcal{X}_{a+b+2}}-\frac {a+b+2}2)}] = \mathbb{E}[e^{z(\overline{\mathcal{X}_{a}}-\frac {a}2)}]\mathbb{E}[e^{z(\overline{\mathcal{X}_{b}}-\frac {b}2)}].$

As a result, $\overline{\mathcal{X}_{a+b+2}}-\frac {a+b+2}2=\overline{\mathcal{X}_{a}}-\frac {a}2+\overline{\mathcal{X}_{b}}-\frac {b}2$ in all moments, and $\{\overline{\mathcal{X}_d} - \frac{d}{2}\}_{d \geq 0}$ form a nice CMRVS.
\end{proof}
We present a combinatorial proof that $\{\overline{\mathcal{X}_d} - \frac{d}{2}\}_{d \geq 0}$ form a nice CMRVS in Appendix~\ref{sec:combproof}, and another proof in Corollary~\ref{cor:chiyproofCMRVS} being the consequence of the Hirzebruch $\chi_y$-genus introduced in Section~\ref{sec:Chern}.
\begin{remark}
If we keep the condition $t \leq d$, we find that 
\[
\mathbb{E}\!\left[e^{z\mathcal{X}_d}\right]
\;=\;\sum_{k\ge0}\frac{\mathbb{E}\!\left[\mathcal{X}_d^k\right]}{k!}z^k
\;=\; \sum_{0\leq t \leq d}\left[x^t\right]\frac{x^{d+2}}{(\log(1+x))^{d+2}(1+x)} (e^z-1)^t
\] is the truncated sum of the first $d+1$ terms.
\end{remark}

\begin{definition} \label{def:Ekprimepoly}
For integers $d,k\ge0$ set
\[
E'_k(d)\;:=\;\mathbb{E}\!\Big[\big(\overline{\mathcal{X}_d}-\tfrac d2\big)^{k}\Big].
\]
\end{definition}
\begin{corollary} \label{cor:equalwhendgeqk-1}
The functions $E'_k(d)$ have the following properties.
\begin{enumerate}
    \item The functions $\{E'_k(d)\}_{k\geq0}$ form a nice CMFS.
    \item We have the exponential generating function $$\sum_{k\ge0}E'_k(d)\frac{z^{k}}{k!} \;=\;\Bigg(\sum_{t\ge0}\frac{z^{2t}}{4^t(2t+1)!}\Bigg)^{d+2}.$$
    \item When $d \geq k-1$, we have $E'_k(d) = \mathbb{E}\![(\mathcal{X}_d-\tfrac d2)^{k}]$.
\end{enumerate} 
\end{corollary}
\begin{proof}

(1) By Definition \ref{def:central_sequence}, the condition that the sequence of random variables $\{\overline{\mathcal{X}_d} - \frac{d}{2}\}_{d \geq 0}$ forms a nice CMRVS is equivalent to the condition that the sequence of functions $\{E'_k(d)\}_{k\geq0}$ forms a nice CMFS. 

(2) Theorem \ref{thm:Ekpoly} tells us that
$$\sum_{k\ge0}E'_k(d)\frac{z^{k}}{k!}
\;=\;\mathbb{E}\!\big[e^{z(\overline{\mathcal{X}_d}-\frac d2)}\big]
\;=\;\Big(\frac{e^{z/2}-e^{-z/2}}{z}\Big)^{d+2}
=\;\Bigg(\sum_{t\ge0}\frac{z^{2t}}{4^t(2t+1)!}\Bigg)^{d+2}.$$

(3) When $t \leq k \leq d+1$, $\mathbb E[\mathcal{X}_d^t] =\mathbb E[\overline{\mathcal{X}_d}^t]$, and $E'_k(d) = \mathbb{E}\![(\overline{\mathcal{X}_d}-\tfrac d2)^{k}]=\mathbb{E}\![(\mathcal{X}_d-\tfrac d2)^{k}]$.
\end{proof}

In particular, $E'_{2k+1} \equiv 0$. For the even terms, $E'_{2k}(d)$ is a polynomial in $d$ of degree $k$. Its leading coefficient equals $(1/12)^k(2k-1)!!$, and we have the following special values (with the convention $0^0=1$).
\[
E'_{2k}(-2)=0^k,\qquad E'_{2k}(-1)=\frac{1}{(2k+1)4^k},\qquad E'_{2k}(0)=\frac{1}{(2k+1)(k+1)}.
\]

The first few polynomials are
\[
E'_0(d)=1,\;
E'_2(d)=\frac{d+2}{12},\;
E'_4(d)=\frac{(d+2)(5d+8)}{240},\;
E'_6(d)=\frac{(d+2)(35d^2+98d+72)}{4032},\dots
\]

One might wish to apply Theorem~\ref{thm:generalineq} directly to the whole nice CMFS $\{E'_k(d)\}$. Unfortunately, $\mathbb{E}[\overline{\mathcal{X}_d}^k]=\mathbb{E}[\mathcal{X}_d^k]$ need not hold when $d<k-1$. For example,
\[
\mathbb{E}\!\Big[\big(\mathcal{X}_2-\tfrac{2}{2}\big)^{4}\Big]\;=\;\tfrac13
\;>\; E'_4(2)\;=\;\tfrac{3}{10},
\]
so the first deviation occurs already at $k=4$ when $d=2$. This is why in Corollary \ref{cor:ineq_moments} we can only apply the functions $\{E'_0(d) = 1, E'_1(d) = 0, E'_2(d) = \frac{d+2}{12}\}$ up to $k=2$.

This motivates the following conjecture.

\begin{conjecture} \label{conj:bestboundisboolean}
Let $\mathrm{M}$ be a matroid of rank $d+1$, and let $k\ge0$ be an integer. Then the central moments of $\mathcal{X}_{\mathrm{M}}$ are bounded by the corresponding Boolean matroids: $$\mathbb{E}\!\Big[\big(\mathcal{X}_{\mathrm{M}}-\tfrac{d}{2}\big)^k\Big]\;\le\; \mathbb{E}\!\Big[\big(\mathcal{X}_{d}-\tfrac{d}{2}\big)^k\Big].$$
\end{conjecture}
\subsection{An inductive polynomial bound}\label{sec:sharpness}

Though Conjecture~\ref{conj:bestboundisboolean} remains unsolved, we can salvage the inequality for the fourth central moment by adding a correction term. Concretely, if $f:\mathbb{Z}_{\ge0}\to\mathbb{R}$ satisfies $f(a+b+2)\ge f(a)+f(b)$ and $f(d)+E'_4(d)\ge\mathbb{E}[(\mathcal{X}_d-\frac{d}{2})^4]$, then Theorem~\ref{thm:generalineq} applies to
\[
\{E'_0(d),\;E'_1(d),\;E'_2(d),\;E'_3(d),\;E'_4(d)+f(d)\}.
\]
The function $f$ has to satisfy $$f(0) \;\geq\; -\tfrac{1}{16}, \quad f(1) \;\geq\; -\tfrac{1}{10}, \quad f(2) \;\geq\; \tfrac{1}{30}, \quad f(d) \;\geq\; 0 \; \text{ for } d> 2.$$ One can find the ``best'' polynomial bound $f(d)=(d+2)/120$ in the sense that it is minimized when $d \rightarrow \infty$, which yields \[
\mathbb{E}\!\Big[\big(\mathcal{X}_{\mathrm{M}}-\tfrac{d}{2}\big)^4\Big]\;\le\; \mathbb{E}\!\Big[\big(\mathcal{X}_{d}-\tfrac{d}{2}\big)^4\Big]\;+\;\frac{d+2}{120}.
\]

We can iterate the method to inductively construct a CMFS of arbitrary length. The algorithm is stated as follows.

\begin{algorithm} \label{alg:findthebestpoly}
\caption{Constructing the CMFS}
\begin{algorithmic}[1] 
\State Start with the sequence of polynomials consisting of a single term $\{f_0(x)\equiv 1\}$, and a polynomial $h(z) = 1$.
\State Assume inductively that $k$ is an even integer (starting with base case $k=0$), and that we have a sequence of polynomials $\{f_0(x), \ldots, f_{k}(x)\}$ such that for all nonnegative integers $t \leq k$ and rank $d+1$ matroids we have $\mathbb{E}[(\mathcal{X}_{\mathrm{M}} - \frac{d}{2})^t] \leq f_t(d)$. Furthermore, we have an even polynomial $h(z)$ of degree $k$ such that $$h(z)^{x+2} \;= \;\sum_{i=0}^k \frac{f_i(x)}{i!}z^i \;+\; (\text{higher order terms of $z$})$$ for all positive integers $x$. In particular, $f_t(x) = 0$ if $t$ is odd, and $f_t(x)$ is a polynomial of degree $t$ if $t$ is even.
\State We can find a specific polynomial $g(x)$ such that $$g(a+b+2) \;=\; g(a) \;+\; g(b) \;+\; 
\sum_{i=2}^{k}\binom{k+2}{i}\, f_i(a)\, f_{k+2-i}(b),$$ where $g(x)$ is a degree $k+2$ polynomial.
\State Find the smallest rational number $C$ such that $\mathbb{E}[(\mathcal{X}_{d} - \frac{d}{2})^{k+2}] \leq g(d) + C(d+2)$ for all $d \geq 0$. 
\State The polynomials $f_{k+1}(x) = 0$, $f_{k+2}(x) = g(x) + C(x+2)$, and $\tilde{h}(z) = h(z) + \frac{C}{(k+2)!}z^{k+2}$ satisfy the induction hypothesis. Return to Step 2 with $k\leftarrow k+2$, $h(z) \leftarrow \tilde{h}(z)$, and the extended CMFS. 
\end{algorithmic}
\end{algorithm}

\begin{proposition}
Algorithm \ref{alg:findthebestpoly} is well-defined. The algorithm produces the smallest power series $h(z) \in \mathbb{Q}[[z]]$ in terms of lexicographical order from the lowest degree, such that if we write $$h(z)^{x+2} \;= \;\sum_{i=0}^\infty \frac{f_i(x)}{i!}z^i,$$ then for any matroid $\mathrm{M}$ of rank $d+1$, we have the inequality $$\mathbb{E}\!\Big[\big(\mathcal{X}_{\mathrm{M}}-\frac{d}{2}\big)^k\Big] \; \leq \; f_k(d).$$
\end{proposition}

\begin{proof} We split the proof into the following steps.

\medskip\noindent\textbf{The existence of $g(x)$ in Step 3.}
We define $g(x)$ to be the coefficient of $z^{k+2}$ in $h(z)^{x+2}(k+2)!$, which is a polynomial of degree $k+2$. More concretely, $$h(z)^{x+2} \;= \;\sum_{i=0}^k \frac{f_i(x)}{i!}z^i \;+\; \frac{g(x)}{(k+2)!}z^{k+2}\;+\;(\text{higher order terms of $z$}).$$ 

Comparing the coefficients of $z^{k+2}$ on both sides of $h(z)^{a+2} h(z)^{b+2} = h(z)^{a+b+2}$ yields the desired identity.

\medskip\noindent\textbf{The existence of $C$ in Step 4.}
By Corollary~\ref{cor:equalwhendgeqk-1}, when $d \geq k +1$, $\mathbb{E}[(\mathcal{X}_{d} - \frac{d}{2})^{k+2}]=E'_{k+2}(d)$ is the coefficient of $z^{k+2}$ in $$\Bigg(\sum_{t\ge0}\frac{z^{2t}}{4^t(2t+1)!}\Bigg)^{x+2}(k+2)! \;=\; \left(1+\frac{1}{24}z^2 + \frac{1}{1920}z^4+\cdots\right)^{x+2} (k+2)!.$$

The polynomial $h(z)$ takes the form $1 + \frac{1}{24}z^2 + \frac{1}{1152}z^4 + \cdots$. Therefore, regardless of the higher-order terms, the coefficient of $z^{k+2}$ in $h(z)^{d+2}(k+2)!$ strictly dominates $E'_{k+2}(d)$ for sufficiently large $d$. Consequently, $\mathbb{E}[(\mathcal{X}_{d} - \frac{d}{2})^{k+2}] \leq g(d)$ for sufficiently large $d$, and there is a smallest $C\in\mathbb{Q}$ such that $\mathbb{E}[(\mathcal{X}_{d} - \frac{d}{2})^{k+2}] \leq g(d) + C(d+2)$ for all $d$.

\medskip\noindent\textbf{The induction hypothesis in Step 5.}
Theorem \ref{thm:generalineq} then tells us $\mathbb{E}[(\mathcal{X}_{\mathrm{M}} - \frac{d}{2})^{t}] \leq f_t(d)$ for all rank $d$ matroids $\mathrm{M}$. Furthermore, by expanding $$\left(h(z) + \frac{C}{(k+2)!}z^{k+2}\right)^{x+2} \,= \;\sum_{i=0}^k \frac{f_i(x)}{i!}z^i \;+\; \frac{g(x)}{(k+2)!}z^{k+2}\;+\;\frac{C(x+2)}{(k+2)!}z^{k+2} \;+\;(\text{H.O.T}),$$ we conclude the result by setting $f_{k+2}(x)=g(x) + C(x+2)$. It is immediate from the construction that the final result $h(z)\in\mathbb{Q}[[z]]$ is the smallest in terms of lexicographical order from the lowest degree.
\end{proof}

An analogous argument gives bounds for the augmented Chow polynomial.

\begin{remark} \label{rem:boundsfordifferential}
Similarly, for a rank $d+1$ matroid $\mathrm{M}$, we can obtain upper bounds for the ratio between the $k$-th derivative and the original Chow polynomial by evaluating $\mathbb{E}[\binom{\mathcal{X}_{\mathrm{M}}}{k}]$. However, for the third derivative, the polynomial bound from Boolean matroids: $$\chowpoly_{\mathrm{M}}^{(3)}(1) \leq \frac{(d-1)(d-2)^2}{8}\chowpoly_{\mathrm{M}}(1)$$ already fails at $d=0$, and we need to add correction terms earlier. Furthermore, when $d \geq 2k-2$, the expectation $\mathbb{E}[\binom{\mathcal{X}_{\mathrm{M}}}{k}]$ is a positive sum of the central moments $\mathbb{E}[(\mathcal{X}_{\mathrm{M}}-\frac{d}{2})^t]$ with $t \leq k$, and the inequalities of this form are generally weaker. Nonetheless, it provides new information starting from $k=4$ and $d=5$.
\end{remark}
\subsection{Bounding central moments via the normal distribution}\label{sec:normaldist}

As shown in Corollary \ref{cor:ineq_chow_coefficient}, the bound for the second moment is sharp, and we can derive bounds for higher central moments. However, obtaining a precise bound for the $k$-th moment for general $k$ via this approach becomes intractable. In this subsection, we prove that the central moments of a matroid are bounded by a normal distribution, as stated in Theorem~\ref{thm:main-theorem}.
\begin{lemma}
The sequence of normal distributions $\{\mathcal{N}(0, \frac{d+2}{12})\}_{d\ge0}$ is a nice CMRVS.
\end{lemma}
\begin{proof}
By the additivity of variances for independent normal distributions,
\[
\mathcal{N}\left(0, \,\frac{a+2}{12}\right)\;+\;\mathcal{N}\left(0, \,\frac{b+2}{12}\right) \;=\;\mathcal{N}\left(0, \,\frac{(a+b+2)+2}{12}\right).
\]
\end{proof}

The rest of Section \ref{sec:normaldist} is devoted to proving the following theorem.
\begin{theorem}\label{thm:main-theorem}
Let $\mathrm{M}$ be a matroid of rank $d+1$, and let $k\ge0$ be an integer. Then the central moments of $\mathcal{X}_{\mathrm{M}}$ are bounded by the corresponding central moments of the normal distribution: 
\begin{equation} \label{eq:normalineq}
    \mathbb{E}\!\Big[\big(\mathcal{X}_{\mathrm{M}}-\tfrac{d}{2}\big)^k\Big]\;\le\; \mathbb{E}\!\Big[\mathcal{N}\big(0, \tfrac{d+2}{12}\big)^k\Big].
\end{equation}
\end{theorem}

By Corollary~\ref{cor:generalineq}, it suffices to prove the inequality for the Boolean case. When $k$ is odd, both sides are $0$, so we only need to consider the case where $k$ is even. Set $k= 2t$, and we want to show that for $d,t\ge0$,
\begin{equation}\label{eq:centralmomentnormal}
\mathbb{E}\!\Big[\big(\mathcal{X}_d-\tfrac d2\big)^{2t}\Big]
\;\le\; \mathbb{E}\!\Big[\mathcal{N}\big(0, \tfrac{d+2}{12}\big)^{2t}\Big]
\;=\; \Big(\frac{d+2}{12}\Big)^t(2t-1)!! .
\end{equation}

\begin{example}
For $t=1$ the right-hand side equals $\frac{d+2}{12}$, which matches $E'_2(d)$; this explains the chosen variance. For $t=2$ one has
\[
\Big(\frac{d+2}{12}\Big)^2\cdot 3\; =\; \frac{(d+2)(5d+10)}{240} \;=\; E'_4(d)\;+\;\frac{d+2}{120},
\]
agreeing with the best bound given in Section~\ref{sec:sharpness}.
\end{example}
\begin{proof}[Proof of Theorem \ref{thm:main-theorem}]
We prove \eqref{eq:centralmomentnormal} in three cases.

\medskip\noindent\textbf{Case (1). $d\ge 2t-1$.} 
When $d\ge 2t-1$, the $(2t)$-th central moment of $\mathcal{X}_d$ equals that of the formal random variable $\overline{\mathcal{X}_d}$. Hence it suffices to compare the even part of the two moment generating functions:
\[
\mathbb{E}\!\big[e^{z\mathcal{N}(0, \frac{d+2}{12})}\big]
\;=\; \exp\!\Big(\frac{z^2}{24}\Big)^{d+2} \quad \text{and} \quad \mathbb{E}\big[e^{z(\overline{\mathcal{X}_d}-\frac d2)}\big]\;=\; \Big(\frac{e^{z/2}-e^{-z/2}}{z}\Big)^{d+2}.
\]
Both power series are even with nonnegative coefficients. A straightforward termwise comparison shows
\[
\exp\!\Big(\frac{z^2}{24}\Big)\;\succeq\;\frac{e^{z/2}-e^{-z/2}}{z},
\]
i.e., every coefficient of the left-hand series is at least the corresponding coefficient of the right-hand series. Convolving these nonnegative coefficient series $(d+2)$ times preserves the coefficient-wise inequality, and the desired inequality of even moments follows.

\medskip\noindent\textbf{Case (2). $d<2t-1$ and $2t<200$.}
This is a finite set of pairs $(d,t)$. We verify \eqref{eq:centralmomentnormal} by direct computational verification; this confirms that equality occurs only for $(d,t)=(2,2)$. The pseudocode for this procedure (Algorithm \ref{alg:verify}) is provided in Appendix~\ref{sec:tedious}.

\medskip\noindent\textbf{Case (3). $d<2t-1$ and $t\geq100$.}
We use a crude bound based on Eulerian numbers. First, we record the elementary bound.

\begin{lemma}
For all integers $n\ge1$ and $0\le i\le n-1$ one has $A(n,i)\le (i+1)^n$.
\end{lemma}
\begin{proof}
A permutation with $i$ descents splits into $i+1$ consecutive increasing runs; assigning each of the $n$ entries to one of these blocks gives at most $(i+1)^n$ possibilities.
\end{proof}

Using $A(d+1,i)\le (i+1)^{d+1}$ and symmetry of the summand we obtain
\[
\mathbb{E}\!\Big[\big(\mathcal{X}_d-\tfrac d2\big)^{2t}\Big]
\;=\;\frac{1}{(d+1)!}\,\sum_{i=0}^d \,A(d+1,i)\,\Big(\frac{d-2i}{2}\Big)^{2t}
\;\le\; \frac{2}{(d+1)!}\sum_{0\le 2i<d}(i+1)^{d+1}\Big(\frac{d-2i}{2}\Big)^{2t}.
\]
Write $d+1=2ct$ with $0<c<1$. By AM--GM, $(i+1)^c(2ct-2i-1) \leq \frac{2}{c}(\frac{c(i+1)+c^2t-ci-c/2}{c+1})^{c+1}$. Therefore, the central moment is bounded by
\[
\frac{2ct}{4^t(2ct)!}\Big(\big(\tfrac{c^2t+c/2}{c+1}\big)^{c+1}\cdot\frac{2}{c}\Big)^{2t} \;= \;\frac{2ct}{(2ct)!}\Big(\big(\tfrac{ct+1/2}{c+1}\big)^{c+1}\cdot c^c\Big)^{2t}.
\]
Consequently it suffices to prove the inequality
$$\frac{2ct}{(2ct)!}\Big(\big(\tfrac{ct+1/2}{c+1}\big)^{c+1}\cdot c^c\Big)^{2t}\; \leq \; \left(\frac{d+2}{12}\right)^{t}(2t-1)!! \;= \;\frac{(2ct+1)^t(2t)!}{24^tt!},$$ or equivalently,  $$\frac{2ct (t!)}{(2t)!(2ct)!}\;\leq\;\left(\frac{(c+1)^{2c+2}(2ct+1)}{24c^{2c}(ct+\frac{1}{2})^{2c+2}}\right)^{t}.$$

For any $n \geq 1$, $n!$ is between $\sqrt{2\pi n} (\frac{n}{e})^n$ and $\sqrt{2\pi n} (\frac{n}{e})^ne^{\frac{1}{12n}}$. So it is enough to prove $$\frac{2cte^{\frac{1}{12t}}}{\sqrt{8ct\pi}}\;\leq\;\left(\frac{(c+1)^{2c+2}(2ct+1)(\frac{2t}{e})^2(\frac{2ct}{e})^{2c}}{24c^{2c}(ct+\frac{1}{2})^{2c+2}(\frac{t}{e})}\right)^{t}\;=\;\left(\frac{2^{2c}(c+1)^{2c+2}}{3e^{2c+1}(c+\frac{1}{2t})^{2c+1}}\right)^{t}.$$

The numerical comparison is straightforward for all $t\geq100$ and $0<c<1$. We supply the full estimate and numeric bounds in Appendix~\ref{sec:tedious}, Lemma~\ref{lem:ugly}.

Combining the three cases proves \eqref{eq:centralmomentnormal} for all $d,t$, and hence Theorem~\ref{thm:main-theorem}.
\end{proof}
The inequality \eqref{eq:normalineq} is an equality when \(k=0\) or odd.  For small even \(k\), the equality cases are understood: \(k=2\) is treated in Corollary~\ref{cor:ineq_chow_coefficient}, and for \(k=4\) Lemma~\ref{lem:flatinSi} shows that equality occurs only when the simplification of \(\mathrm{M}\) is the Boolean matroid \(\mathrm{U}_3\).  For larger even \(k\) there are no equality cases.

\begin{remark}
The bounds provided by Theorem~\ref{thm:main-theorem} weaken as \(k\) grows.  For example, one can verify that the first couple of terms of $h(z)$ in Algorithm~\ref{alg:findthebestpoly} are
\[
\mathbb{E}\!\big[e^{z(\mathcal{X}_d-\tfrac d2)}\big]
\;\le\;\Big(1+\tfrac{1}{24}z^2+\tfrac{1}{1152}z^4-\tfrac{1}{46080}z^6+\cdots\Big)^{d+2},
\]
hence only the first three terms coincide, and for any matroid \(\mathrm{M}\) of rank \(d+1\) we have the improved bound
\[
\mathbb{E}\!\Big[\big(\mathcal{X}_{\mathrm M}-\tfrac d2\big)^6\Big]
\;\le\; \frac{(d+2)(5d^2+20d+6)}{576},
\]
which is better than the bound \(\tfrac{5(d+2)^3}{576}\) coming from Theorem~\ref{thm:main-theorem}.
\end{remark}
\subsection{The naive bound and the numbers of flags of flats} \label{sec:reversed}

So far, we have focused on the upper bounds of the $k$-th central moments for $\mathcal{X}_{\mathrm M}$. It turns out that lower bounds can be established via the following elementary observation.

\begin{lemma}[The naive bound] \label{lem:naiveineq}
Let $\mathcal X$ be an integer valued random variable supported on $\{0,\dots,d\}$ and assume it is symmetric:
\[
\Pr(\mathcal X=s)\;=\;\Pr(\mathcal X=d-s)\quad(0\le s\le d).
\]
Then for every even integer $k = 2t > 0$,
\[
\mathbb{E}\!\Big[\big(\mathcal X-\tfrac{d}{2}\big)^k\Big]\;\geq\;
\begin{cases}
0, & \text{if $d$ is even},\\[1pt]
\frac{1}{2^k}, & \text{if $d$ is odd}.
\end{cases}
\]
Furthermore, for every integer $k$ with $2k - 2 \geq d$,
\[
\mathbb{E}\!\left[\binom{\mathcal{X}}{k}\right]\;\geq\; 0.
\]
In particular, the inequality holds for $\mathcal X=\mathcal X_{\mathrm M}$ for any rank \(d+1\) matroid \(\mathrm M\).
\end{lemma}

We now show that this lower bound is asymptotically sharp. We do this by exhibiting an explicit family of rank-$(d+1)$ matroids whose Chow polynomials concentrate on the middle indices. The key computation is an exact formula for the number of flags of flats in the projective geometry \(\mathrm{PG}(d,q)\), and we compute the Chow polynomial from the number of flags of flats.

Let \(M_q:=\mathrm{PG}(d,q)\) denote the rank $d+1$ matroid realizing the arrangement consisting of all nonzero vectors of \((\mathbb F_q)^{d+1}\).  

\begin{proposition} \label{prop:degreeMq}
Let $J = \{j_1, \ldots, j_m\} \subseteq [d]$ be a set with $j_1 < \cdots < j_m$. $N_J(M_q)$ (cf.\ Definition~\ref{def:NJflagsofflats}) equals the product of Gaussian binomials
\[
\prod_{i=1}^m \binom{\,d+1-j_{i-1}\,}{\,j_i - j_{i-1}\,}_q,
\qquad j_0:=0,
\]
where we use the standard notation \[\binom{m}{r}_q \;=\;\frac {(1-q^{m})(1-q^{m-1})\cdots (1-q^{m-r+1})}{(1-q)(1-q^{2})\cdots (1-q^{r})} \] for the Gaussian binomial coefficient. Define $j_{m+1} = d+1$. The degree of $N_J(M_q)$ in $q$ is $$E(J) \;= \; \frac{(d+1)^2}{2} \;-\; \frac{1}{2} \sum_{i=1}^{m+1} (j_i - j_{i-1})^2.$$ 
\end{proposition}

\begin{proof}
Each step that increases dimension from \(j_{i-1}\) to \(j_i\) chooses a \((j_i-j_{i-1})\)-dimensional subspace of an ambient space of dimension \(d+1-j_{i-1}\); the number of choices equals \(\binom{d+1-j_{i-1}}{j_i-j_{i-1}}_q\). Multiplying over the steps gives the displayed product. 

The polynomial \(\binom{m}{r}_q\) has degree $r(m-r)$. The degree of $N_J$ then equals $$\sum_{i=1}^m (d+1-j_i) (j_i - j_{i-1})\; =\; j_m(d+1)\;-\;\sum_{i=1}^m j_i (j_i - j_{i-1}) \;=\; \frac{(d+1)^2}{2} \;-\; \frac{1}{2} \sum_{i=1}^{m+1} (j_i - j_{i-1})^2.$$
\end{proof}
The coefficients of the Chow polynomial can be computed by the number of flags of flats, and \cite[Corollary 3.3.3]{Simplicial} gives a precise basis for the Chow ring from flags of flats. In particular, 
\begin{proposition}[{\cite[page 13]{FY_Chow}}] 
\label{prop:chow_poly_FMSV}
For every matroid $\mathrm{M}$,
\[
\chowpoly_{\mathrm{M}}(x)\;=\sum_{\varnothing=F_0\subsetneq F_1\subsetneq\cdots\subsetneq F_m}
\prod_{i=1}^{m} \frac{x\big(1-x^{\operatorname{rk}(F_i)-\operatorname{rk}(F_{i-1})-1}\big)}{1-x}.
\]   
\end{proposition}
Consider any flag of flats terminating in $F_m \subsetneq E$. This flag can be extended by optionally appending $F_{m+1} = E$. Grouping these pairs of flags together yields an extra factor of $$\frac{x\big(1-x^{(d+1)-\operatorname{rk}(F_{m}) - 1}\big)}{1-x}\; + \;1 \;= \;\frac{x\big(1-x^{(d+2)-\operatorname{rk}(F_{m}) - 1}\big)}{x(1-x)}.$$

\begin{notation}[Block decomposition associated to an index set] \label{notation:block}
Let $J=\{j_1<j_2<\cdots<j_m\}\subseteq [d]$ be a strictly increasing set of indices.  
We associate to $J$ a sequence of \emph{block sizes} $(n_1,\dots,n_{m+1})$ defined by
\[
n_1 = j_1,\quad n_2 = j_2-j_1,\quad \dots,\quad n_m = j_m-j_{m-1},\quad n_{m+1}=d+2-j_m,
\]
and for $J=\varnothing$ we set $m=0$ and $n_1=d+2$. 
Thus, in all cases, we have 
\[
\sum_{i=1}^{m+1} n_i\; =\; d+2.
\]
Whenever we refer to an index set $J$, we also implicitly refer to the associated block sizes $(n_1,\dots,n_{m+1})$ defined above.
\end{notation}

This gives rise to another version of the proposition. 
\begin{proposition}[Another version of Proposition \ref{prop:chow_poly_FMSV}]\label{prop:chowpolyfof}
For every matroid $\mathrm{M}$,
\[
\chowpoly_{\mathrm{M}}(x)\;=\;\frac{1}{x}\sum_{J \subseteq [d]} N_J
\prod_{i=1}^{m+1} \frac{x\big(1-x^{n_i-1}\big)}{1-x}. \quad\text{(cf.\ Notation convention \ref{notation:block})}
\]   
\end{proposition}

Note that if some $n_i = 1$, it would not contribute to the computations. Under the constraint that $n_i > 1$, $E(J)$ (cf.\ Proposition~\ref{prop:degreeMq}) reaches maximum when $J = \{2, 4, \ldots, 2\lfloor \tfrac{d}{2}\rfloor\}$. So, as $q \to \infty$, the Chow polynomial for $M_q$ is asymptotically given by $$\begin{cases}
    x^{\frac{d}{2}} &, \;d \text{ is even,}\\
    x^{\frac{d-1}{2}}(1+x) &, \;d \text{ is odd.}\\    
\end{cases}$$
Therefore, the naive bound given in Lemma~\ref{lem:naiveineq} is strict.
\begin{corollary} \label{cor:naiveboundformatroids}
For integers $d, k$ with $k > 0$, the optimal constant $C_{d, k}$ such that the inequality
\[
\mathbb{E}\!\Big[\big(\mathcal X_{\mathrm M}-\tfrac{d}{2}\big)^k\Big]\;\geq\;
C_{d,k}
\] holds for all matroids $\mathrm M$ of rank $d+1$ is
\[
C_{d, k} \;=\;
\begin{cases}
0, & \text{if $d$ is even or $k$ is odd},\\[1pt]
\frac{1}{2^k}, & \text{else}.
\end{cases}
\]
Furthermore, for $2k -2 \geq d$, $D_{d,k} = 0$ is the optimal constant such that the inequality
\[
\mathbb{E}\!\left[\binom{\mathcal{X}_{\mathrm{M}}}{k}\right]\;\geq\; D_{d,k}
\] holds for all matroids $\mathrm M$ of rank $d+1$.
\end{corollary}

This approach also suggests that central moment inequalities can be translated into linear inequalities on the flag counts $N_J$. Fix a single flag type \(J= \{j_1,\dots,j_m\}\) whose block sizes are \(n_1,\dots,n_{m+1}\). The contribution of one such flag to the Chow polynomial is
\[
P_J(x)\;=\;\frac{1}{x}\prod_{i=1}^{m+1} f_{n_i}(x),
\qquad\text{where } f_n(x)\;:=\;x+x^2+\cdots+x^{n-1}.
\]

Let $\mathcal F_J$ be the formal random variable whose probability generating function is \(P_J(x)/P_J(1)\). The equation $\chowpoly_{\mathrm M}(x) = \sum_J N_JP_J(x)$ implies
\[
\chowpoly_\mathrm{M}(1)\cdot \mathbb{E}\!\Big[\big(\mathcal X_{\mathrm M}-\tfrac d2\big)^k\Big]
\;=\; \sum_J N_J\,P_J(1)\,\mathbb{E}\!\Big[\big(\mathcal F_J-\tfrac d2\big)^k\Big],
\] where $\chowpoly_\mathrm{M}(1) = \sum_J N_JP_J(1) $. For example, applying Corollary \ref{cor:ineq_chow_coefficient} yields the following.
\begin{corollary} \label{cor:ineqfof}
Let $\mathrm M$ be a matroid of rank $d+1$. Then
\[
\sum_{J\subset\{1, \ldots, d\}}\prod_{i=1}^{m+1} (n_i-1)
           \left(\sum_{i=1}^{m+1}\frac{n_i(n_i-3)}{2}\right) N_J \;\leq\; 0,
\]
with equality precisely when $d=1$ or when the simplification of $\mathrm{M}$ is $\mathrm{U}_{d+1}$ with $d>0$.
\end{corollary}

\begin{proof}
Compare the coefficient of \(N_J\) on both sides of
\[
\sum_JN_J\, P_J(1) \cdot \frac{d+2}{12} \;\geq\; \sum_JN_J\, P_J(1)\cdot \mathbb{E}\!\Big[\big(\mathcal X_{\mathrm M}-\tfrac d2\big)^2\Big]
\;=\; \sum_J N_J\,P_J(1)\,\mathbb{E}\!\Big[\big(\mathcal F_J-\tfrac d2\big)^2\Big].
\]
$\mathbb{E}[(\mathcal F_J-\tfrac d2)^2] = \operatorname{Var}(\mathcal F_J)$ is the variance, which, by linearity of expectation, equals the sum of variances of each block: $$\operatorname{Var}(\mathcal F_J)\;=\;\sum_{i=1}^{m+1} \frac{1}{n_i-1}\sum_{j=1}^{n_i-1}\left(j - \frac{n_i}{2}\right)^2 \;= \;\sum_{i=1}^{m+1} \frac{n_i(n_i-2)}{12}.$$ 

Therefore, the coefficient equals
$$P_J(1)\left(\sum_{i=1}^{m+1} \frac{n_i(n_i-2)}{12} - \frac{d+2}{12}\right) \;=\; \prod_{i=1}^{m+1} (n_i-1)
           \left(\sum_{i=1}^{m+1}\frac{n_i(n_i-3)}{2}\right),$$
and we conclude the inequality.
\end{proof}
Fix a $d > 0$. For a set $J = \{j_1,\dots,j_m\} \subseteq \{1, \ldots, d\}$ define the number $U_J = N_J(\mathrm{U}_{d+1})$, which is equal to
\[
U_J \;:=\; \frac{n_{m+1}(d+1)!}{\prod_{i=1}^{m+1} n_i!}.
\]

Corollary~\ref{cor:ineqfof} implies the nontrivial equality $$\sum_{J}\prod_{i=1}^{m+1} (n_i-1)\left(\sum_{i=1}^{m+1}\frac{n_i(n_i-3)}{2}\right) U_J \;=\; 0,$$ which can also be shown directly by a generating-function argument. The rest of Section~\ref{sec:reversed} records an attempt to prove the inequality from the combinatorial side.

\begin{lemma} \label{lem:elemonoineq}
Let $\mathrm{M}$ be a matroid of rank $d+1$. For nested index sets \(J\subseteq J'\), we have the monotonicity bound:
\[
\frac{N_J}{U_J}\;\leq\; \frac{N_{J'}}{U_{J'}}.
\]
\end{lemma}
\begin{proof}
Write the rank-jumps associated to $J$ as $n_1,\dots,n_{m+1}$, and those for $J'$ as a refinement
\[
n_i\; = \;m_{i,1}\;+\;\cdots\;+\;m_{i,\ell_i}.
\]

Fix a flag of flats $\mathcal{F}$ of type $J$.  
Each interval $F_{i-1}<F_i$ contains at least $n_i!$ maximal chains refining that interval.  
Among these, the number of chains whose rank-jumps occur in blocks of sizes $m_{i,1},\dots,m_{i,\ell_i}$ is at least
\[
\frac{n_i!}{m_{i,1}!\cdots m_{i,\ell_i}!}.
\]

Doing this independently in each interval shows that every $J$-flag admits at least $\frac{U_{J'}}{U_J}$ refinements of type $J'$.  Therefore, $N_{J'} \ge N_J\cdot\frac{U_{J'}}{U_J}.$
\end{proof}
\begin{conjecture} \label{conj:coneofineq}
We conjecture that the inequality in Corollary \ref{cor:ineqfof}
\[
\sum_{J\subset\{1, \ldots, d\}}\prod_{i=1}^{m+1} (n_i-1)
           \left(\sum_{i=1}^{m+1}\frac{n_i(n_i-3)}{2}\right) N_J \;\leq\; 0
\]
is a nonnegative linear combination of the monotonicity inequalities established in Lemma~\ref{lem:elemonoineq}.
\end{conjecture}

\begin{remark}
We verified Conjecture \ref{conj:coneofineq} by computer for all ranks \(d+1\) with \(d\le 19\). 
\end{remark}
\subsection{Connection to $\gamma$-coefficients of the Chow polynomial}
\label{sec:Rlabelings} 
By \cite[Theorem 1.8]{Chowpoly}, the Chow polynomial of any matroid \(\mathrm M\) of rank \(d+1\) is \emph{\(\gamma\)-positive}: there exist nonnegative integers
\(\gamma_0,\dots,\gamma_{\lfloor d/2\rfloor}\) (with \(\gamma_0=1\)) such that
\[
\chowpoly_{\mathrm M}(x)
\;=\;\sum_{t=0}^{\lfloor d/2\rfloor}\gamma_t\,x^{t}(x+1)^{d-2t}.
\]
In particular \(\chowpoly_{\mathrm M}(1)=\sum_t\gamma_t\,2^{\,d-2t}\), and the leading term \((x+1)^d\) provides the trivial bound used below.

\begin{corollary}
\label{cor:binomial_bound}
Let \(\mathrm M\) be a matroid of rank \(d+1\). For every integer \(k\ge0\),
\[
\mathbb{E}\!\Big[\big(\mathcal X_{\mathrm M}-\tfrac d2\big)^k\Big]
\;\le\;
\mathbb{E}\!\Big[\big(\mathcal{B}(d,\tfrac12)-\tfrac d2\big)^k\Big],
\]
where \(\mathcal{B}(d,\tfrac12)\) denotes a \(\mathrm{Binomial}(d,\tfrac12)\) variable.
\end{corollary}

\begin{proof}
Writing \(\chowpoly_{\mathrm M}\) as the combination
\[
\frac{\chowpoly_{\mathrm M}(x)}{\chowpoly_{\mathrm M}(1)}
\;=\;\sum_{t\ge0} w_t\,(x+1)^{d-2t}x^t,\qquad
w_t\;:=\;\frac{\gamma_t\,2^{\,d-2t}}{\sum_{s\ge0}\gamma_s\,2^{\,d-2s}},
\]
shows that \(\mathcal X_{\mathrm M}\) is a mixture of binomial distributions \(t+\mathcal{B}(d-2t,\tfrac12)\) with weights \(w_t\). Since they all have mean \(d/2\) and the central moment \(\mathbb{E}[(\mathcal{B}(n,\tfrac12)-n/2)^k]\) is nondecreasing in \(n\), we conclude the claim.
\end{proof}
\begin{remark} 
For a fixed $d$, the normal distribution bounds in Theorem \ref{thm:main-theorem} are stronger when $k$ is small,
while the binomial bounds in Corollary \ref{cor:binomial_bound} are stronger for larger $k$. This is because the binomial distribution has a bounded support and its $k$-th central moment can never exceed $(d/2)^k$. Meanwhile, the normal distribution's $k$-th moment grows super-exponentially with $(k-1)!!$ as stated in \eqref{eq:centralmomentnormal}. 
\end{remark}
Furthermore, each central moment can be written as
\[
\mathbb{E}\!\Big[\big(\mathcal X_{\mathrm M}-\tfrac d2\big)^k\Big]
=\frac{\sum_{t\ge0}\gamma_t\,2^{\,d-2t}\,
      \mathbb{E}\!\big[\big(\mathcal{B}(d-2t,\tfrac12)-\tfrac{d-2t}{2}\big)^k\big]}
     {\sum_{t\ge0}\gamma_t\,2^{\,d-2t}}.
\]

In particular, bounds on the central moments translate into constraints on the \(\gamma_t\).  For example, applying Corollary \ref{cor:ineq_chow_coefficient} yields the following.
\begin{corollary} \label{cor:ineqgamma}
Following the notation above, the variance bound implies the inequality $$
\frac{\sum_{t\ge0}2^{\,d-2t-2}\left(d-2t\right)\,\gamma_t}
{\sum_{t\ge0} 2^{\,d-2t}\,\gamma_t}\; \leq \;\frac{d+2}{12},$$
or equivalently,
$$
\sum_{t\ge0}2^{-2t}\left(d-3t-1\right)\,\gamma_t \;\leq\; 0,$$ with equality precisely when $d=1$ or when the simplification of $\mathrm{M}$ is $\mathrm{U}_{d+1}$ with $d>0$.
\end{corollary}
\begin{remark}
\cite[Theorem 1.1]{Rlabel} gives a combinatorial interpretation of the coefficients~$\gamma_t$ in terms of maximal flags of flats with specific descent sets with respect to any R-labelling. We can also express the $\gamma$-coefficients in terms of the central moments, and translate them into Chern numbers using the techniques in Section~\ref{sec:Chern}.
\end{remark}
\subsection{Constraints on roots and the real-rootedness conjecture}
\label{sec:real_rootedness}
The following real-rootedness conjecture was independently proposed by Ferroni-Schr\"{o}ter \cite[Conjecture 8.18]{FS_valuative_24} and Huh-Stevens \cite[Conjecture 4.1.3 and 4.3.3]{Stevens_realrooted}.
\begin{conjecture} \label{conj:realroots}
    The Chow polynomial of a matroid is real-rooted.
\end{conjecture}
Suppose the Chow polynomial factors as $\chowpoly_{\mathrm M}(x)=\prod_{i=1}^d (x+z_i)$ with $z_i \in \mathbb{C}$. Set
\[
\mu_i\;:=\;\frac{1}{1+z_i}\qquad(1\le i\le d).
\]
For each \(i\), let \(\mathcal Z_i\) be the (formal) Bernoulli random variable with probability generating function
\((x+z_i)/(1+z_i)\), so \(\mathbb{E}[\mathcal Z_i]=\mu_i\). Then
\[
\mathcal X_{\mathrm M}\;=\;\sum_{i=1}^d \mathcal Z_i,
\]
and for every \(k\ge1\) the central moment
\[
\mathbb{E}\!\Big[\big(\mathcal X_{\mathrm M}-\tfrac d2\big)^k\Big]
\;=\;\mathbb{E}\!\Big[\big(\sum_{i=1}^d (\mathcal Z_i-\mu_i)\big)^k\Big]
\]
is a polynomial in the power sums $W_t:=\sum_{i=1}^d \mu_i^t$ with $1\le t\le k$.

Consequently, any upper (or lower) bound on the central moments yields linear or polynomial constraints on the \(W_t\)'s.
\begin{example}
For small \(t\) the relations are especially simple:
\begin{itemize}
 \item \(W_1=\sum_i\mu_i=\mathbb{E}[\mathcal X_{\mathrm M}]=d/2.\)
 \item The second central moment is $W_1 - W_2$, thus
   \[
   0\;\le \;W_1-W_2\;=\;\mathbb{E}\!\Big[\big(\mathcal X_{\mathrm M}-\frac d2\big)^2\Big]\;\le\; \frac{d+2}{12}, \quad \text{and }\;\frac{5d-2}{12}\;\le\; W_2\;\le\; \frac d2.
   \]
 \item The third central moment gives the linear relation
   \[
   0=\mathbb{E}\!\Big[\big(\mathcal X_{\mathrm M}-\frac d2\big)^3\Big]=W_1-3W_2+2W_3,
   \]
   hence
   \[
   W_3=\frac{1}{2}(3W_2-W_1),\quad \text{and }\;\frac{3d-2}{8}\le W_3\le \frac d2.
   \]
 \item The fourth central moment is a quadratic polynomial in \(W_1,\dots,W_4\). The inequality \(0\le\mathbb{E}[(\mathcal X_{\mathrm M}-\tfrac d2)^4]\le \frac{(d+2)^2}{48}\) 
   yields another constraint.
\end{itemize}    
\end{example}
Real-rootedness of $\chowpoly_{\mathrm M}$ is equivalent to real-rootedness of $f(x) =\prod_i (x+\mu_i)$. While calculating the power sums $W_t$ is generally insufficient to prove real-rootedness directly, our bounds on the central moments impose strict geometric constraints on the distribution of the roots $z_i$ in the complex plane. Should the real-rootedness conjecture hold, the inequalities in the example above provide necessary bounds that the roots of any matroid Chow polynomial must satisfy. The inequalities for the power sums $W_t$ could potentially provide information about the sign of $f(x)$ for certain values of $x$, allowing one to isolate real roots between points with alternating signs.

%% file: 4-c_top.tex
\section{Chern classes of matroids}
\label{sec:Chern}

\subsection{Interpretation as inequalities of Chern numbers}
The coefficients of the Chow polynomial of a matroid $\mathrm M$ and the total Chern class $c(T_{\mathrm M})$ are intimately linked. In this section, we establish the connection between them, allowing us to translate inequalities of Chow coefficients directly into inequalities of Chern numbers.

A compact K\"{a}hler manifold \(X\) is said to be of \emph{pure type} (or Hodge--Tate type) if \(h^{p,q}(X)=0\) for \(p\neq q\).
If a matroid \(\mathrm M\) of rank \(d+1\) is realizable over \(\mathbb C\) by a linear subspace \(L \subseteq \mathbb{C}^E\), the corresponding wonderful compactification \({W}_L\) is obtained from \(\mathbb P^d\) by an iterated sequence of blow-ups along linear centers. Since \(\mathbb P^d\) and the linear blow-up centers are of pure type by induction, so is \({W}_L\), and the cycle map is an isomorphism in this setting (cf.\ \cite[Theorem 7.31]{Voisin_Hodge_1} and \cite[Example 19.1.11]{Fulton_Intersection}). In particular,
\[
a_p(\mathrm M)\;=\;\dim_\mathbb{C} H^{2p}({W}_L,\mathbb{C})\;=\;h^{p,p}({W}_L)
\;=\;(-1)^p\chi\big({W}_L,\Omega^p_{{W}_L}\big).
\]

To relate these topological invariants to Chern numbers, we recall the Hirzebruch \(\chi_y\)-genus. For a smooth complex projective variety \(X\) of complex dimension \(d\), we set
\[
\chi^p(X)\;:=\;\chi\big(X,\Omega^p_X\big)\;=\;\sum_{q=0}^d(-1)^q h^{p,q}(X),\qquad
\chi_y(X)\;:=\;\sum_{p=0}^d \chi^p(X)\,y^p.
\]
Writing the Taylor expansion of \(\chi_y(X)\) at \(y=-1\) as
\begin{equation}\label{eq:expansion_chi_y}
\chi_y(X)\;=\;\sum_{p=0}^d \xi_p(X)\,(y+1)^p,
\end{equation}
the Hirzebruch--Riemann--Roch theorem expresses \(\chi_y(X)\) as an integral (equivalently, as a polynomial in the Chern roots \(x_i\)):
\begin{equation}\label{eq:chi_y_genus}
\chi_y(X)
\;=\;\int_X \prod_{i=1}^d \frac{x_i(1+ye^{-x_i})}{1-e^{-x_i}} .
\end{equation}

It is classical (see e.g.,\ \cite{NR_Prym_75, LW_Kahler_90, Salamon_Kahler_96}) that each coefficient \(\xi_p(X)\) can be written as a linear combination of Chern numbers. For a modern and systematic exposition of these expansions, we refer to \cite{LiPing}.

For $k \geq 0$, setting $0^0=1$, we write 
\[
h_k(X)\;:=\; \sum_{p=0}^d (-1)^p\,\chi^p(X) \left(p - \frac{d}{2}\right)^k.
\] 
Each $h_k(X)$ is a $\mathbb{Q}$-linear combination of $\xi_i(X)$, and we have the following explicit formulas.

\begin{lemma}[{\cite[Lemma 2.1]{LiPing}}]\label{lem:hirzeformula}
Let \(X\) be a compact K\"{a}hler manifold of complex dimension \(d\). When $k$ is odd, $h_k(X) = 0$; the first several explicit expressions of $h_k$ for even $k$ are given below (provided that $d \geq k$):
\[
\begin{aligned}
 h_0(X)\;&=\;c_d(X),\\[3pt]
 h_2(X)\;&=\;\frac{d}{12}c_d(X)\;+\;\frac{1}{6}c_1c_{d-1}(X),\\[3pt]
 h_4(X)\;&=\;\frac{d(5d-2)}{240}c_d(X)\,+\,\frac{5d-2}{60}c_1c_{d-1}(X) \,+\, \frac{c_1^2+3c_2}{30}c_{d-2}(X)\, -\, \frac{c_1^3-3c_1c_2+3c_3}{30}c_{d-3}(X),\\
 \dots
\end{aligned}
\]
\end{lemma}

For matroids, the Chern numbers are defined in \cite[Definition 3.2]{Tbundle}. We simplify notation by identifying top-degree classes with their degrees; specifically, $c_d$ denotes the Chern number $\deg(c_d(\mathrm{M}))$, and products such as $c_1 c_{d-1}$ or $c_k \alpha^{d-k}$ denote the corresponding intersection numbers.

\begin{proposition}
Let $\mathrm{M}$ be a matroid of rank $d+1$. The central moments $\mathbb{E}[(\mathcal{X}_{\mathrm{M}} - \frac{d}{2})^k]$ can be expressed in terms of Chern numbers exactly as the ratio $h_k/h_0$.
\end{proposition}
\begin{proof}
When $\mathrm{M}$ is realizable via $L\subseteq \mathbb{C}^E$, we have $\chi_{-y}({W}_L)=\chowpoly_{\mathrm{M}}(y)$, and 
\[
h_k({W}_L) \;=\; h_0({W}_L) \, \mathbb{E}\!\Big[\big(\mathcal{X}_{\mathrm{M}} - \tfrac{d}{2}\big)^k\Big] .
\]
Since both the Chow polynomial and the Chern numbers are valuative, this equality extends to arbitrary matroids.
\end{proof}

As a direct application, inequalities for the central moments can be translated into inequalities for Chern numbers. For example, applying Corollary \ref{cor:ineq_chow_coefficient} yields the following strikingly simple inequality.

\begin{corollary}
 \label{cor:ineq_of_chern_numbers}
 Let $\mathrm{M}$ be a matroid of rank $d+1$. Then the following inequality of Chern numbers holds:
\begin{equation}
 \label{eq:ineq_of_chern_numbers}
c_1c_{d-1}\;-\;c_d\;\leq \;0,
 \end{equation}
with equality precisely when $d=1$ or when the simplification of $\mathrm{M}$ is $\mathrm{U}_{d+1}$ with $d>0$.
\end{corollary}

\begin{proof}
Combining Lemma~\ref{lem:hirzeformula} and Corollary~\ref{cor:ineq_chow_coefficient}, we have 
\[
\mathbb{E}\!\Big[\big(\mathcal{X}_\mathrm{M} - \frac{d}{2}\big)^2\Big] \;=\; \frac{d}{12} \;+\; \frac{c_1c_{d-1}}{6c_d} \;\leq \;\frac{d+2}{12}.
\] 
Since $c_d = \chowpoly_{\mathrm{M}}(1)$ is positive, the result follows.
\end{proof}

\begin{remark} \label{rem:Fulton-Lazarsfeld}
We can interpret \eqref{eq:ineq_of_chern_numbers} geometrically. The Fulton--Lazarsfeld inequalities assert that the Schur classes of a nef vector bundle are nonnegative; in particular, $c_1c_{d-1} \ge c_d$. For ample vector bundles, this inequality is strict: $c_1c_{d-1} > c_d$. (See \cite[Theorem 2.5]{NEFtangentbundlesLFtype} for the original statement, and \cite{LiPingnefineq} for a modern exposition).

Corollary~\ref{cor:ineq_of_chern_numbers} establishes the opposite bound $c_1 c_{d-1} \le c_d$. This has two immediate geometric consequences for realizable matroids. First, it implies that the tangent bundle can never be ample. This is consistent with the classical fact that a smooth projective variety has ample tangent bundle if and only if it is projective space. 

Second, for non-Boolean matroids where the strict inequality $c_1 c_{d-1} < c_d$ holds, the tangent bundle cannot be nef. While this aligns with the geometry of the blow-up construction of $W_L$, our inequality strengthens this observation by providing a purely numerical obstruction in $K$-theory: \emph{no} vector bundle with the same $K$-class can be ample, and in the non-Boolean case, none can be nef. 

It is worth noting that the opposite direction of the Fulton--Lazarsfeld inequalities is also false, in general. For the Boolean matroid of rank 5, we have the Chern numbers $c_2^2 = 130 > c_1c_3 = c_4 = 120$ (see Proposition~\ref{prop:booleanintersection}), and the value of $c_2^2 - c_4$ of a rank 5 matroid can be positive or negative depending on the specific matroid.

A natural direction for future work is to define a combinatorial notion of nef vector bundles for matroids, enabling a systematic study of the nefness of the tangent class $T_{\mathrm{M}}$. Finally, in dimension $d=2$, the inequality $c_1^2 \le c_2$ is strictly stronger than the Bogomolov--Miyaoka--Yau inequality $c_1^2 \le 3c_2$ for surfaces of general type.
\end{remark}

Furthermore, the polynomials $E'_k(d)$ (cf.\ Definition~\ref{def:Ekprimepoly}) predict the numeric value of $h_k(X_E)$, where $X_E$ is the permutahedron variety of dimension $d \geq k-1$. Certain Chern numbers for $X_E$ are straightforward to compute, which we detail in Appendix~\ref{sec:Chernandchiy}. For example, we find:
\begin{equation} \label{eq:chi_y_Cherncomputation}
\begin{aligned}
 h_4(X_E)\;&=\;\frac{d(5d-2)}{240}c_d\;+\;\frac{5d-2}{60}c_1c_{d-1} \;+\; \frac{c_1^2+3c_2}{30}c_{d-2}\; -\; \frac{c_1^3-3c_1c_2+3c_3}{30}c_{d-3}\\
 &=\; c_d\left(\frac{d(5d-2)}{240} \;+\; \frac{5d-2}{60}\;+\;\frac{1}{30}\cdot\frac{3d+37}{12}\;-\;\frac{1}{30}\cdot\frac{3d+1}{12}\right)\\
 &= \;\frac{(d+2)(5d+8)}{240}c_d \;=\; E'_4(d)\,c_d \;=\; E'_4(d)(d+1)!.
\end{aligned}
\end{equation}

This coincides with the fact that $h_k(X_E) = E'_k(d)h_0(X_E)$ when $d \geq k$. Note that the coefficient of $c_d$ in $h_4(X_E)$ is exactly $E'_4(d-2)$; this is not a coincidence and emerges structurally from \eqref{eq:chi_y_genus}. Further discussion is provided in Appendix~\ref{sec:Chernandchiy}.

\begin{remark}
Let $X$ be a smooth projective complex variety of dimension $d$ such that 
\[
(-1)^p\chi^p(X) \;=\; \sum_q (-1)^{p+q}h^{p,q}(X) \geq 0, \quad \text{for all } 0\leq p\leq d.
\]
This occurs, for example, when all odd Betti numbers vanish. Set $\mathcal{X}$ as the random variable such that \[
  \Pr(\mathcal{X}=k) \;=\;\frac{(-1)^k\chi^k(X)}{\sum_{i=0}^d (-1)^i\chi^i(X)},\qquad k=0,\dots,d.
\]  Then the naive inequalities about $\mathbb{E}[(\mathcal{X}-\frac{d}{2})^k]$ and $\mathbb{E}[\binom{\mathcal{X}}{k}]$ from Lemma~\ref{lem:naiveineq} immediately yield Chern class inequalities for $X$. For instance, when $k = 2$, we have:
\begin{align*}
2c_1c_{d-1} + d c_d \;&=\; 3\sum_{i=0}^d (d-2i)^2 (-1)^i \chi^i\;\geq \;0 \quad \text{($d$ even), \; and} \\
2c_1c_{d-1} + (d-3) c_d \;&= \;3\sum_{i=0}^d \left((d-2i)^2 - 1\right) (-1)^i \chi^i\;\geq \;0\quad \text{($d$ odd).}
\end{align*}

Furthermore, the ratios in Lemma~\ref{lem:naiveineq} are sharp. Consider $Y_n$, the blow-up of $\mathbb{P}^2$ at $n$ points (with Poincar\'{e} polynomial $x^2 + (n+1)x + 1$). For even $d$, take $X = (Y_n)^{\frac{d}{2}}$; for odd $d$, take $X = \mathbb{P}^1 \times (Y_n)^{\frac{d-1}{2}}$. As $n \rightarrow \infty$, the Poincar\'{e} polynomial concentrates at the middle degree(s), and the moment lower bounds are attained asymptotically. 
\end{remark}

\subsection{Multiplication by \texorpdfstring{$\alpha$}{alpha} and truncation} \label{subsec:alpha_truncation}

In this subsection, we derive further inequalities for the Chern classes of $\mathrm{M}$ by analyzing their intersection with the class $\alpha$.

Multiplication by the class \(\alpha\) corresponds geometrically to passing to the \emph{principal truncation} of the matroid (see \cite[Theorem~3.2.3]{Simplicial}, noting that $\alpha = h_E$). For a matroid $\mathrm{M}$ of rank $d + 1$ on the ground set $E$, its truncation $T(\mathrm{M})$ is the matroid on the same ground set $E$ defined by the rank function: 
\[
\operatorname{rk}_{T(\mathrm{M})}(S) = \min(\operatorname{rk}_{\mathrm{M}}(S), \;d), \quad \text{for all } S\subseteq E.
\]

Therefore, for a matroid \(\mathrm{M}\) realized by $L \subseteq \Bbbk^E$, the vanishing locus of a general section of \(\alpha\) inside the wonderful variety \(W_L\) is precisely the wonderful variety associated to a realization $L'$ of $T(\mathrm{M})$. The normal bundle of \(W_{L'}\) in \(W_L\) is the line bundle represented by the divisor class \(\alpha\). The tangent--normal exact sequence
\[
0 \longrightarrow T_{W_{L'}} \longrightarrow T_{W_L}\big|_{W_{L'}}
\longrightarrow N_{W_{L'}/W_L} \longrightarrow 0
\]
demonstrates that the total Chern class of \(T_{W_L}\big|_{W_{L'}}\) is obtained from that of \(T_{W_{L'}}\) by adjoining a single additional Chern root equal to \(\alpha\).

To clarify this behavior and extend it to non-realizable matroids, recall from Theorem~\ref{thm:Tangentbundle} that $c(T_\mathrm{M})$ is the product of:
\[
1+Z_1, \; \ldots, \; 1+Z_d, \; 1+\alpha, \; 1+\alpha - Z_d, \; \ldots, \;1+ \alpha - Z_d - \cdots - Z_1.
\]

Upon multiplying by $\alpha^k$, the elements $Z_{d-k+1}, \ldots, Z_d$ all become zero in the Chow ring. Consequently, $c(T_\mathrm{M}) \cdot \alpha^k$ becomes the product of:
\[
1+Z_1, \; \ldots, \; 1+Z_{d-k},\; (1+\alpha)^{k+1}, \; 1+ \alpha - Z_{d-k} \; \ldots, \;1+ \alpha - Z_{d-k} - \cdots - Z_1.
\] 
This corresponds exactly to the total Chern class of the tangent bundle of the rank-$(d-k+1)$ truncation, multiplied by $k$ additional factors of $(1+\alpha)$, since adjoining each extra Chern root equal to $\alpha$ multiplies the total Chern class by $(1+\alpha)$. Thus, the Chern classes of the new rank satisfy the standard relation:
\[
c^{\text{new}}_k \;=\; c_k \;+\; \alpha\,c_{k-1}.
\]

By Proposition \ref{prop:valuativefunctions}, any product of $Z_i$ and $\alpha$ can be expressed as a linear combination of the flag counts $N_J$. Moreover, this expansion behaves stably under truncation:

\begin{lemma}
Suppose that for a rank $d+1$ matroid $\mathrm{M}$, we have 
\[
Z_{t_1}(\mathrm{M}) \cdots Z_{t_d}(\mathrm{M}) \;=\; \sum_J s_J N_{J}(\mathrm{M}).
\]
Then, for a rank $d + 1 + k$ matroid $\mathrm{M}'$, we have 
\[
Z_{t_1}(\mathrm{M}') \cdots Z_{t_d}(\mathrm{M}') \cdot \alpha^{k} \;=\; \sum_J s_J N_{J}(\mathrm{M}').
\]
That is, multiplying by $\alpha$ does not affect the coefficients of $N_J$.
\end{lemma}
\begin{proof}
Recall from \cite[Lemma 3.13]{Tbundle} that 
\[
\deg(\alpha^{d-k}\beta^{k}) = (-1)^k\sum_{J \subseteq \{1, \ldots, k\}} (-1)^{|J|} N_{J},
\]
where the coefficients do not depend on $d$. Following the proof of \cite[Proposition 3.16]{Tbundle}, since the general coefficients are determined by expanding the products of $\deg(\alpha^{i}\beta^{j})$, the lemma holds.
\end{proof}

For example, for $k \leq d$, we can compute $c_k \alpha^{d-k}$ and $c_1c_{k-1} \alpha^{d-k}$ as linear combinations of $N_J$. We define the coefficients $f_{k, J}(d)$ and $g_{k, J}(d)$ via the expansions:
\[
c_{k} \alpha^{d-k} = \sum_{J \subseteq [d]} f_{k, J}(d) N_J, \quad \text{and} \quad c_1c_{k-1}\alpha^{d-k} = \sum_{J \subseteq [d]} g_{k, J}(d) N_J.
\]
The recursive relations $c_k \mapsto c_k + \alpha c_{k-1}$ and $c_1c_{k-1} \mapsto (c_1+ \alpha) (c_{k-1} +\alpha c_{k-2})$ yield the recurrences:
\begin{equation}\label{eq:recursion}
\begin{cases}
f_{k, J}(d+1)&= \;f_{k, J}(d)\; + \;f_{k - 1, J}(d),\\
g_{k, J}(d+1) &= \;g_{k, J}(d)\; + \;g_{k-1, J}(d) \;+ \;f_{k-1, J}(d)\; +\; f_{k-2, J}(d).
\end{cases}
\end{equation}
As a result, a subset $J$ contributes a non-zero term only if $J \subseteq [k]$.

\begin{proposition}\label{prop:coeff-formula} 
For any matroid $\mathrm M$ of rank $d+1$, we can write $c_d$ and $c_1c_{d-1}$ as
\[
c_d \;=\; \sum_J f_{d, J}(d)\,N_J,\qquad
c_1c_{d-1} \;=\; \sum_J g_{d, J}(d)\,N_J,
\]
where (cf.\ Notation \ref{notation:block})
\[
f_{d, J}(d) \;=\; \prod_{i=1}^{m+1} (n_i-1),
\qquad
g_{d, J}(d) \;=\; f_{d, J}(d)
\left(\sum_{i=1}^{m+1}\frac{n_i(n_i-3)}{2} + 1\right).
\]
\end{proposition}
\begin{proof}
This explicitly follows from the algebraic expansion computed in Corollary~\ref{cor:ineqfof}.
\end{proof}

\begin{remark}
The formula for \(f_{d, J}(d)\) can also be verified geometrically by tracking the topological Euler characteristics through the iterated blow-ups that define the wonderful compactification.
\end{remark}

We can now compute $f_{k, J}(d)$ and $g_{k, J}(d)$ in general via \eqref{eq:recursion}.

\begin{corollary} \label{cor:cherncoefficientfof}
Suppose $J = \{j_1, \dots, j_m\} \subseteq [k]$. Let $A = \prod_{i=1}^m n_i$ and $B = \sum_{i=1}^m n_i(n_i-3)/2$. Note that these products and sums are taken only up to index $m$, omitting the final block $n_{m+1} = d+2-j_m$ which depends on $d$. Then:
\[
\begin{cases}
f_{k, J}(d) \;&=\; A \binom{d - j_m + 1}{k - j_m},\\[2pt]
g_{k, J}(d) \;&=\; A \left(B\binom{d - j_m + 1}{k - j_m} + (d - j_m + 1)\binom{d - j_m + 1}{k - j_m-1}\right).
\end{cases}
\]
\end{corollary}
\begin{proof}
This follows by a straightforward induction on $d$ using the recurrences in \eqref{eq:recursion}.
\end{proof}

To bound these counts, we compare matroids under weak maps.

\begin{definition}
We say there exists a \emph{rank-preserving weak map} $\varphi:\mathrm{M}_1 \rightarrow \mathrm{M}_2$ between two matroids $\mathrm{M}_1$ and $\mathrm{M}_2$ if the following conditions hold (cf.\
\cite[Conjecture 3.32]{Chowpoly}):
\begin{enumerate}
\item $\mathrm{M}_1$ and $\mathrm{M}_2$ share the same ground set,
\item $\operatorname{rk}(\mathrm{M}_1) = \operatorname{rk}(\mathrm{M}_2)$, and
\item all independent sets of $\mathrm{M}_2$ are independent in $\mathrm{M}_1$.
\end{enumerate}
\end{definition}

For example, if $\mathrm{M}$ is a rank $d+1$ matroid on the ground set $[n]$, there always exists a rank-preserving weak map $\varphi:\mathrm{U}_{d+1, n} \rightarrow \mathrm{M}$. This structure is respected by the flag counts:

\begin{proposition}[{\cite[Proposition 8.19]{catvaluative}}]
Let $\mathrm M_1$ and $\mathrm M_2$ be matroids of rank $d+1$, and suppose there exists a rank-preserving weak map $\varphi:\mathrm M_1\to\mathrm M_2$. Fix $J=\{j_1<\cdots<j_m\}\subseteq [d]$ and let $\Phi_J(\mathrm M_i)$ denote the $\mathbb Q$-vector space spanned by all flags of flats of $\mathrm M_i$ whose ranks are $j_1,\dots,j_m$. Then the natural map $\Phi_J(\varphi):\Phi_J(\mathrm M_1)\rightarrow\Phi_J(\mathrm M_2)$, defined by
\[
(F_1<\cdots<F_m)\;\longmapsto\;
\begin{cases}
(\overline{F_1}<\cdots<\overline{F_m}) &\text{if the closures have ranks }j_1,\dots,j_m,\\[4pt]
0 &\text{otherwise},
\end{cases}
\]
is surjective.
\end{proposition}

As an immediate consequence, we obtain the monotonicity of $c_k(\mathrm{M})\alpha^{d-k}$ with respect to weak maps.

\begin{corollary} \label{cor:weakmapineq}
Let $\mathrm M_1$ and $\mathrm M_2$ be matroids of rank $d+1$, and suppose there exists a rank-preserving weak map $\varphi:\mathrm M_1\to\mathrm M_2$. Then for $0\leq k \leq d$: 
 \[
 c_k(\mathrm{M}_1)\alpha^{d-k}\;\geq\; c_k(\mathrm{M}_{2})\alpha^{d-k}.
 \]
\end{corollary}
\begin{proof}
By writing $N_J(\mathrm M_i)=\dim\Phi_J(\mathrm M_i)$, the surjectivity above guarantees $N_J(\mathrm{M}_1) \geq N_J(\mathrm{M}_2)$. The result follows since the coefficients $f_{k,J}(d)$ of $N_J$ in the expansion of $c_{k}(\mathrm{M}_i) \alpha^{d-k}$ are all nonnegative.
\end{proof}

Furthermore, we establish the following natural absolute lower bounds.

\begin{lemma} \label{lem:topchernineq}
Let $\mathrm{M}$ be a rank $d+1$ matroid. For every $J\subseteq[d]$, we have: 
\[
N_{J}(\mathrm{U}_{d+1}) \;\leq\; N_{J}(\mathrm{M}).
\]
\end{lemma}
\begin{proof}
Fix a basis $B$ of $\mathrm M$ (so $|B|=d+1$). The closure map $S \mapsto \overline{S}$ sends subsets of $B$ (which act as flats in the Boolean matroid $\mathrm U_{d+1}$) injectively to flats in $\mathrm M$ of the same rank, proving the claimed inequality.
\end{proof}

\begin{corollary}
Let $\mathrm{M}$ be a rank $d+1$ matroid. For $0\leq k \leq d$, we have:
\[
c_k(\mathrm{M})\alpha^{d-k}\;\geq\; c_k(\mathrm{U}_{d+1})\alpha^{d-k} \;= \;\frac{(d+1)!}{(d-k+1)!}.
\]
\end{corollary}
\begin{proof}
The inequality is directly implied by Lemma~\ref{lem:topchernineq}. It remains to compute $c_k(\mathrm{U}_{d+1})\alpha^{\,d-k}$. For the permutahedron variety, the total Chern class equals $\prod_{i=1}^d(1+Z_i(\mathrm{U}_{d+1}))$.

Since $Z_i(\mathrm{U}_{d+1}) \alpha^{d-k} = 0$ for $i > k$, it follows that:
\[
c_k(\mathrm{U}_{d+1})\alpha^{d-k}\;=\; \left(\prod_{i=1}^k Z_i(\mathrm{U}_{d+1})\right)\alpha^{d-k} \;=\; N_{1, \ldots,k}(\mathrm{U}_{d+1})\; = \; \frac{(d+1)!}{(d-k+1)!}.
\]
\end{proof}

For example, setting $k=2$ yields: 
\begin{equation} \label{eq:smallcaseTd}
\alpha^{d-2} \cdot (c_1^2 + c_2) \;=\; \frac{(d+1)(3d+2)}{2},
\end{equation}
which can also be verified via Corollary \ref{cor:cherncoefficientfof}. 

\begin{corollary} \label{cor:miyaoka-yau-alpha}
Let $\mathrm{M}$ be a rank $d+1$ matroid with $d \geq 2$. Then: 
\[
\big((d+2)c_2 - 2dc_1^2\big)\alpha^{d-2} \;=\; (3d+2)\left(N_2 - \binom{d+1}{2}\right) \; \geq \;0.
\]
\end{corollary}
\begin{proof}
Corollary \ref{cor:cherncoefficientfof} gives:
\[
\begin{cases}
c_2\,\alpha^{d-2} &=\; \binom{d+1}{2}+N_2,\\[2pt]
c_1^2\,\alpha^{d-2} &=\; (d+1)^2 - N_2.
\end{cases}
\]
The final inequality follows from $N_2(\mathrm{M}) \geq N_2(\mathrm{U}_{d+1}) = \binom{d+1}{2}$.
\end{proof}

\begin{remark} \label{rem:myineq}
This inequality is stronger than the standard Miyaoka--Yau inequality for intersections with $\alpha$ (see, for example, \cite{miyaokayau, miyaokayau2}). Since $c_2\alpha^{d-2} > 0$, we have:
\[
dc_1^2\alpha^{d-2}\;\leq \;\frac{d+2}{2}c_2\alpha^{d-2} \;<\;2(d+1)c_2\alpha^{d-2}.
\]
However, it is generally false that $(dc_1^2-2(d+1)c_2)D^{d-2} \leq 0$ for an arbitrary nef divisor $D$. For example, for a rank $4$ matroid, the intersection $(3c_1^2-8c_2)\beta$ may be strictly positive.
\end{remark}

We conclude this section by posing a natural question regarding higher-degree constraints:
\begin{question} \label{que:ineq_of_two_chern_classes}
Let $\mathrm{M}$ be a matroid of rank $d+1$, and let $0 < k \leq d$. What is the optimal constant $t_{k}(d) \geq 0$ such that
\begin{equation}
c_k\,\alpha^{d-k} \;\geq \;t_{k}(d)\,c_1\,c_{k-1}\,\alpha^{d-k},
\end{equation}
and is this bound precisely achieved by the Boolean matroid $\mathrm{U}_{d+1}$?
\end{question}

Finally, there is also a series of numeric identities regarding the Chern classes and $\alpha$.
\begin{proposition}[{\cite[Corollary 6.7]{Tbundle}}]
Let $\mathrm{M}$ be a rank $d+1$ matroid. For $0\leq k \leq d$, let $\operatorname{td}_{k}$ denote the degree $k$ component of the Todd class. Then: 
\[
\alpha^{d-k} \, \operatorname{td}_{k} \;=\; \deg\left(\left(\alpha + \frac{\alpha^2}{2} + \frac{\alpha^3}{3} + \cdots \right)^{d-k} \frac{1}{1-\alpha}\right).
\]
\end{proposition}
For $k \leq 3$, this aligns perfectly with the results given in Corollary~\ref{cor:cherncoefficientfof}.

%% file: 5-appendix.tex
\appendix 
\section{Combinatorial proof}\label{sec:combproof}
This appendix provides a combinatorial proof of Theorem~\ref{thm:Ekpoly}, showing that the sequence of random variables $\{\overline{\mathcal{X}_d}-\frac{d}{2}\}_{d\geq0}$ is a nice CMRVS. 

\begin{lemma} \label{lem:Ballscount}
Let $r\ge1$ and let $t_1,\dots,t_r$ be positive integers.  Set
\[
S\;:=\;t_1\;+\;\cdots\;+\;t_r.
\]
For nonnegative integers $a,b$, consider two disjoint lines of balls: a left line of length $a$ and a right line of length $b$ (the two lines do not touch).  We count placements of $r$ \emph{ordered} contiguous blocks of chosen balls of sizes $t_1,\dots,t_r$ subject to the condition that:
\begin{itemize}
  \item each block consists of $t_i$ consecutive balls in one of the two lines,
  \item distinct blocks are not adjacent to each other.
\end{itemize}
Let $N(a,b;t_1,\dots,t_r)$ denote the number of such placements.

Then for fixed $t_1,\dots,t_r$, whenever $a,b\geq S-1$ and $a+b = a'+b'$, it holds that
\[
N(a+b;t_1,\dots,t_r)\;:=\;N(a,b;t_1,\dots,t_r)\;=\;N(a',b';t_1,\dots,t_r).
\]
\end{lemma}
\begin{proof}
It is enough to show $N(a,b)=N(a-1,b+1)$ whenever $a-1$ and $b$ are at least $S-1$. We proceed by induction on $r$. When $r = 1$, the statement is immediate.

Write the left sites as $1,\dots,a$ and the right sites as $a+1,\dots,a+b$. Given a valid placement counted by $N(a,b)$, attempt to \emph{move} site $a$ to the right side; this produces a placement in the $(a-1,b+1)$ configuration except in the two exceptional situations described below:
\begin{enumerate}
  \item site $a$ belongs to a block of size $>1$,
  \item site $a$ is a singleton chosen block, and site $a+1$ is also chosen.
\end{enumerate}

Similarly, reversing the move gives a map from placements for $(a-1,b+1)$ to placements for $(a,b)$, with its own exceptional class. Consequently
\[
N(a,b)\;=\;N(a-1,b+1)\;+\;E(a,b)\;-\;E'(a-1,b+1),
\]
where $E(a,b)$ (resp. $E'(a-1,b+1)$) counts the exceptional placements in $(a,b)$ (resp. in $(a-1,b+1)$) that the move does not map forward (resp. backward). The fact $E(a,b)=E'(a-1,b+1)$ comes from induction. In particular, we can record the exact block being violated and remove the block in our induction hypothesis. 
\end{proof}

\begin{lemma} \label{lem:combimomentpf}
Suppose $a,b\ge0$ and $d=a+b+2$. Then for every integer $k\le\min(a,b)+1$ one has
\[
\mathbb{E}\big[\mathcal X_d^k\big]\;=\;\mathbb{E}\big[(1+\mathcal X_a+\mathcal X_b)^k\big],
\]
where $\mathcal X_a$ and $\mathcal X_b$ are independent.
\end{lemma}

\begin{proof}
It is equivalent to proving the identity for falling factorial moments
\[
\mathbb{E}\big[(\mathcal X_d)_{k}\big]\;=\;\mathbb{E}\big[(1+\mathcal X_a+\mathcal X_b)_{k}\big]\; = \;\mathbb{E}\big[(\mathcal X_a+\mathcal X_b)_{k}\big] \;+\;k \; \mathbb{E}\big[(\mathcal X_a+\mathcal X_b)_{k-1}\big]. 
\]

The factorial moment can be written as
\[
\mathbb{E}[(\mathcal X_d)_{k}]
\;=\; k!\sum_{S\subseteq\{1,\dots,d\},\ |S|=k} \Pr\big(\text{all positions in }S\text{ are descents}\big).
\]
Consider the two central positions $M=\{a+1,a+2\}$, every $S$ is of one of three types: (i) $S\cap M=\varnothing$, (ii) $|S\cap M|=1$, or (iii) $|S\cap M|=2$. 

For the left hand side, case (i) corresponds to $\mathbb{E}\big[(\mathcal X_a+\mathcal X_b)_{k}\big]$. For case (ii) we enumerate over the length $t$ of the block containing the central position. It equals $$k!\,\sum_{t=1}^k\left(\frac{1}{(t+1)!}\sum_{t_1 + \cdots + t_r = k-t}\left(\prod_{j=1}^r\frac{1}{(t_j+1)!}2N(a+b-t, t_1, \ldots, t_r)\right)\right).$$ 
This follows from Lemma~\ref{lem:Ballscount} that $N(a+b-t;t_1, \ldots, t_r) = N(a, b-t; t_1, \ldots, t_r)=N(a-t, b; t_1, \ldots, t_r)$. And for case (iii), we enumerate over the length $t$ of the block containing the central position $$k!\,\sum_{t=2}^k\left(\frac{1}{(t+1)!}\sum_{t_1 + \cdots + t_r = k-t}\left(\prod_{j=1}^r\frac{1}{(t_j+1)!}(t-1)N(a+b-t; t_1, \ldots, t_r)\right)\right).$$ 

Setting $t_0 = t-1$, the sum of the above is $$k!\sum_{\substack{t_0+t_1 + \cdots + t_r = k-1, \\ t_0 \geq 0, \,t_1, \ldots,t_r > 0.}}\left(\prod_{j=0}^r\frac{1}{(t_j+1)!}N(a+b-t_0 - 1; t_1, \ldots, t_r)\right).$$ 

We can read $t_0$ as the length of the consecutive descents starting at the first placement (which can be 0). Therefore, this is exactly $k \mathbb{E}[(\mathcal X_a+\mathcal X_b)_{k-1}]$, which completes the proof. 
\end{proof}

\begin{corollary}
The sequence of random variables $\{\overline{\mathcal{X}_d}-\frac{d}{2}\}_{d\geq0}$ is a nice CMRVS.
\end{corollary}
\begin{proof}
Suppose $a,b\ge0$ and $d=a+b+2$. By Lemma~\ref{lem:combimomentpf} we have that for every integer $k\le\min(a,b)+1$, \[
\mathbb{E}\!\Big[\big(\overline{\mathcal{X}_d} - \tfrac{d}{2}\big)^k\Big]\;=\;\mathbb{E}\!\Big[\big(\overline{\mathcal{X}_a} - \tfrac{a}{2}+\overline{\mathcal{X}_b} - \tfrac{b}{2}\big)^k\Big].
\]

For a fixed $k$, the equality holds when $a, b \geq k-1$. The moments of $\overline{\mathcal{X}_d}$ are polynomials over $d$, so this would hold for any $a, b$. Hence we complete the proof.
\end{proof}

\section{Chern numbers and coefficients}\label{sec:Chernandchiy}
In this appendix, we show how we compute the coefficient of $c_d$ in $h_k(X)$ (cf.\ Lemma \ref{lem:hirzeformula}), and how to compute some Chern numbers for the permutahedron variety $X_E$.
\subsection{Coefficients of Chern numbers in the Hirzebruch-\texorpdfstring{$\chi_y$}{\chi_y} genus}
\begin{proposition}\label{prop:appendix_c_d}
Let \(X\) be a compact K\"{a}hler manifold of complex dimension \(d\), and suppose $k \leq d-1$. Then the coefficient of $c_d$ in $h_k(X)$ is $E'_k(d-2)$.
\end{proposition}
Recall $\xi_k(X)$ in the Taylor expansion of $\chi_y$-genus (cf.\ \eqref{eq:expansion_chi_y}). Since
\[
\xi_k(X)\;=\; \frac{(-1)^k}{k!}\sum_{p,q\geq0}^d (-1)^{p+q}h^{p,q}(X)p(p-1)\cdots(p-k+1)
\]
and
\[
E'_k(d)\;=\;\mathbb{E}\!\Big[\big(\overline{\mathcal{X}_d} - \tfrac{d}{2}\big)^k\Big]\;=\;\mathbb{E}\!\Big[\big(\mathcal{X}_d - \tfrac{d}{2}\big)^k\Big],\quad k\leq d,
\]
it suffices to show that the coefficient of $c_d$ in $\xi_k(X)$  is $(-1)^k\mathbb{E}[\binom{\mathcal{X}_{d-2}}{k}]$, where $\mathcal{X}_{d-2}$ is the random variable associated with the uniform matroid $\mathrm{U}_{d-1}$, see Section \ref{sec:moment}.

Recall that complex genera, defined as ring homomorphisms from the complex cobordism ring $\Omega^*_{\mathrm{U}}\otimes \mathbb{Q}$ to an arbitrary commutative ring $R$, correspond one-to-one to monic formal power series with coefficients in $R$, and are rational linear combinations of Chern numbers. For the $\chi_y$-genus (see, e.g.,\ \cite[Section 5.4]{Hirze1996Manifolds}), it takes values in $\mathbb{Z}[y]$ and the corresponding power series is
\[
Q_y(x)\;=\;\frac{x(1+ye^{-x(1+y)})}{1-e^{-x(1+y)}}.
\]
Suppose that $x_1,\ldots,x_d$ are Chern roots of $TX$, then
\[
\chi_y(X)\;=\;\int_X \prod_{i=1}^d Q_y(x_i)\;=\;\sum_{\lambda}b_\lambda(y)c_\lambda[X],
\]
where $\lambda=(\lambda_1\lambda_2\cdots)$ runs over all integer partitions of $d$ and $c_\lambda[X]=\int_Xc_{\lambda_1}(X)c_{\lambda_2}(X)\cdots$ is the Chern number of partition $\lambda$. We can read the coefficient of $c_d$ from the following fact, which is usually attributed to Cauchy (see \cite[\S 1.4]{Hir} or \cite[Lemma 4.1]{24complexgenera_PingLi}):
\begin{lemma}\label{lem:coe_b_d}
    The coefficient $b_d(y)=b_{(d)}(y)$ is determined by
    \[
        1\;+\;\sum_{i=1}^\infty (-1)^ib_i(y)\cdot x^i\; =\; 1\;-\;x\cdot \frac{Q_y'(x)}{Q_y(x)}.
    \]
\end{lemma}
Now we are ready to prove Proposition \ref{prop:appendix_c_d}.

\begin{proof}[Proof of Proposition \ref{prop:appendix_c_d}] On the right hand side of the expansion in Lemma \ref{lem:coe_b_d}, we have
\[
1\;-\;x\cdot \frac{Q_y'(x)}{Q_y(x)}\;=\; x(y+1)\cdot \left(\frac{1}{1-e^{-x(1+y)}}-\frac{1}{1+ye^{-x(1+y)}}\right). 
\]
Let $t=x(1+y)$. Then:
\[
\begin{aligned}
t\left(\frac{e^t}{e^t-1}-\frac{e^t}{e^t+y}\right)\;&=\;\sum_{d=0}^\infty\frac{(-1)^dB_d}{d!}t^d\;-\;t\sum_{k=0}^\infty{(-y)^ke^{-kt}}\\
\;&=\;\sum_{d=0}^\infty\frac{(-1)^dB_d}{d!}t^d\;-\;\sum_{d=1}^\infty\left(\sum_{k=0}^\infty{(-k)^{d-1}\frac{(-y)^k}{(d-1)!}}\right)t^d,
\end{aligned}
\]
where $B_d$ is the $d$-th Bernoulli number with $B_1=1/2$ and $B_{2k+1}=0$ for $k>1$. Let
\[
\mathrm{Li}_s(z)\;=\;\frac{z}{\Gamma(s)}\oint_C\frac{(-w)^{s-1}}{e^w-z}\mathrm{d}w,\quad s\neq1,
\]
be the analytic continuation of $\mathrm{Li}_s(z)=\sum_{i=1}^\infty z^i/i^s$, where $C$ is a path starting from $+\infty$, integrating along the upper half of the real axis up to near the origin, circling around the origin, and then integrating along the lower half of the real axis back to $+\infty$. In other words, $C$ is a contour around the branch cut segment $[1,+\infty]$. This $\mathrm{Li}_s(z)$ is the \emph{polylogarithm function}. In particular,
\[
\zeta(s)\;=\;\mathrm{Li}_s(1) \text{ when } s \neq 1. 
\]
When $s$ is a non-positive integer, $z=1$ is a pole of order $1-s$. Thus, we may evaluate $b_d(y)$ via the following formula:
\[
b_d(y)\;=\;  \frac{(1+y)^d}{(d-1)!}( \mathrm{Li}_{1-d}(-y)-\zeta(1-d)), \quad y \in \mathbb{C}.
\]
Suppose that $s$ is a nonnegative integer and we expand $\mathrm{Li}_{-s}(z)$ near $z=1$ as 
\[
\mathrm{Li}_{-s}(z)\;=\; \frac{(-1)^{s+1}}{(z-1)^{s+1}} \sum_{k=0}^\infty c_{k,s} (z-1)^k.
\]
Since (see, for example, \cite{weisstein_polylogarithm})
\[
\mathrm{Li}_{-s}(z)\;=\;\frac{z A_s(z)}{(1-z)^{s+1}},\quad \text{for } s\in\mathbb{Z}_{\geq0},
\]
we have
\[
c_{k,s}\;=\; \sum_{m=0}^{s-k} A(s,m) \binom{s-m}{k},
\]
where $A_s(z)$ is the Eulerian polynomial and $A(s,m)$ is the Eulerian number. 
Thus,
\[
b_d(y)\;=\; \frac{1}{(d-1)!}  \left(\sum_{k=0}^{d-1} (-1)^{k} c_{k,d-1}(y+1)^k\right)\;+\;\frac{(1+y)^d}{d!}B_d.
\]
and the coefficient in front of $(y+1)^k$ (which is exactly the coefficient of $c_d$ in $\xi_k(X)$) is 
\[
\left\{\begin{aligned}&\frac{(-1)^kc_{k,d-1}}{(d-1)!},\quad k\leq d-1\\ &\frac{B_d}{d!},\quad k=d,\\&0,\quad\text{else}.\end{aligned}\right.
\]
If we set $\Pr(\mathcal{Y}=m)=A(s,m)/s!$, then $c_{k,s}/s! = \mathbb{E}[\binom{s-\mathcal{Y}}{k}]=\mathbb{E}[\binom{\mathcal{Y}}{k}]$. When $s=d-1$, the random variable is $\mathcal{Y}=\mathcal{X}_{d-2}$. This completes the proof. 
\end{proof}
\begin{corollary} \label{cor:chiyproofCMRVS}
The sequence of random variables $\{\overline{\mathcal{X}_d}-\frac{d}{2}\}_{d\geq0}$ is a nice CMRVS.
\end{corollary}

\begin{proof}
Let $X$ and $Y$ be smooth projective varieties of dimensions $a+2$ and $b+2$, respectively. The $\chi_y$-genus is multiplicative under products, so
\[
\chi_y(X\times Y) = \chi_y(X)\chi_y(Y).
\]

Consider the coefficient of $c_{d+2}(X \times Y) = c_{a+2}(X)c_{b+2}(Y)$ in $h_k(X\times Y)$, where $d = a+b+2$ and $k < \min(a,b)$. Since
\[
h_k(X \times Y) = \sum_{i=0}^k \binom{k}{i}h_i(X)h_{k-i}(Y),
\]
Proposition \ref{prop:appendix_c_d} yields
\[
E'_k(d) = \sum_{i=0}^k \binom{k}{i}E'_i(a)E'_{k-i}(b).
\]
The function $E'_k$ is a polynomial, thus the identity extends to all integers $a,b$ with $d=a+b+2$. Therefore, the functions $\{E'_k(d)\}_{k\geq0}$ form a nice CMFS; and by definition, the random variables $\{\overline{\mathcal{X}_d}-\frac{d}{2}\}_{d\geq0}$ form a nice CMRVS.
\end{proof}

\begin{remark}
We can compute the coefficients of arbitrary Chern numbers in $h_k$ by considering products $X \times Y$, where $Y$ is chosen to be a variety with easily computable Chern classes. For example, taking $Y$ to be a product of projective spaces, the product $X \times \mathbb{P}^1$ determines the coefficient of $c_1c_{d-1}$ in $h_k$, while $X \times (\mathbb{P}^1)^2$ and $X \times \mathbb{P}^2$ determine the coefficients of $c_1^2c_{d-2}$ and $c_2c_{d-2}$, respectively, and so on.
\end{remark}

\subsection{Intersection Numbers on the Permutahedron}
Let $X_E$ be the permutahedron variety of dimension $d$. We denote $p_k$ as the $k$-th power sum of the Chern roots. We explain how we compute the Chern numbers for $X_E$ in \eqref{eq:chi_y_Cherncomputation} by the following proposition.
\begin{proposition} \label{prop:booleanintersection}
For any $1 \leq k < d$, the following intersection numbers hold:
\begin{enumerate}
    \item $$c_1^kc_{d-k} \;=\; \frac{\binom{2k+2}{k+1}}{(k+2)!} c_d.$$
    \item $$p_kc_{d-k} \;=\; (-1)^{k-1} \left[\frac{(d-k+1)\binom{2k}{k}}{(k+1)!}-  \frac{(d-k)\binom{2k+2}{k+1}}{(k+2)!} \right]c_d$$
    \item $$c_k c_{d-k} \;= \sum_{\substack{0\leq a \leq \min(k, d-k) \\ 0\leq j\leq a/2}}(-1)^a(\frac{1}{12})^j\binom{d-2a}{k-a}\binom{a-j}{d-2a}\binom{d-j-a}{a-j}\,c_d.$$
\end{enumerate}
\end{proposition}
\begin{proof}[Sketch of proof]

(1) The class $c_1$ equals $\alpha + \beta$. Expanding $c_1^k$ yields a binomial sum:$$ c_1^k c_{d-k} \;=\; \sum_{i+j=k} \binom{k}{i} \frac{1}{(i+1)!(j+1)!}c_d.$$ Applying the identity $\sum_{i} \binom{k}{i} \binom{k+2}{k+1-i} = \binom{2k+2}{k+1}$, the result follows directly.

(2) Consider the power sum $p_k = \sum_{i=1}^d Z_i^k$. The non-vanishing terms in the product $Z_i^kc_{d-k}$ will be of the form $Z_1\cdots Z_x Z_i^j Z_y\cdots Z_d$, where we have $0\leq x < i < y \leq d+1$ with $y-x=j+1$. This $j$ will be either $k$ or $k+1$. The degree of the term is $(-1)^{j-1}\frac{1}{(i-x)!(y-i)!}\binom{y-x-2}{y-i-1}c_d$. Summing over the possible placements of $x, y$, and $i$ yields $(-1)^{j-1}\frac{(d-j+1)\binom{2j}{j}}{(j+1)!}c_d$, and the result follows from combining the cases $j = k$ and $j=k+1$.

(3) To compute the general intersection numbers $c_k c_{d-k}$, we analyze the degrees of monomials of the form $\prod_{i=1}^d Z_i^{t_i}$ where $t_i \in \{0, 1, 2\}$ and $\sum t_i = d$.

Any such non-vanishing monomial can be partitioned into blocks delimited by indices where $t_i = 1$. Consider an interval $[1, j-1]$ between two such indices (or the boundaries). For the monomial to be non-vanishing, the interval length must be even, and the exponents must satisfy the local condition $t_{2y-1} + t_{2y} = 2$ for all $1 \leq y \leq \frac{j-1}{2}$. Let $a_m$ denote the sum of the degrees of all valid exponent assignments in an interval of length $2m$, normalized by $(-1)^m(2m+1)!$. By conditioning on the first even index $2i$ where $t_{2i}=2$, we obtain the recursive relation: 
$$a_m \;=\; \frac{1}{2}a_{m-1} \;+\; \sum_{i=2}^{m}\frac{1}{2^{i-2}\cdot 3}a_{m-i} \;+\; \frac{1}{2^m}.$$ This recurrence simplifies significantly to $a_m = a_{m-1} + \frac{1}{12}a_{m-2}$ for $m > 1$, with initial conditions $a_0 = a_1 = 1$. The corresponding generating function for these blocks is:$$P(x) \;=\; \sum_{m=0}^\infty a_m x^m \;=\; \left(-\frac{1}{12}x^2 - x + 1 \right)^{-1}.$$

The general intersection number is a convolution of these local blocks. Let $a$ be the number of indices where $t_i = 2$ and $b$ be the number of indices where $t_i = 1$, such that $2a + b = d$. The total contribution of these terms is given by the coefficient $G_{a,b}$, defined as:
$$G_{a,b} \;=\; [x^a] \left(-\frac{1}{12}x^2 - x + 1 \right)^{-(b+1)} = \;\sum_{j=0}^{\lfloor a/2 \rfloor} \binom{a-j+b}{b} \binom{a-j}{j} \left(\frac{1}{12}\right)^j,$$ where the identity follows from expanding $P(x)^{b+1} \; = \; (1 - (x + \frac{1}{12}x^2))^{-(b+1)} \;=\; \sum \binom{n+b}{b} (x + \frac{1}{12}x^2)^n$ and extracting the coefficient of $x^a$.

By summing over the possible distributions of exponents that constitute the elementary symmetric polynomials $c_k$, we arrive at the general identity:$$c_k c_{d-k} \;= \;\sum_{2a+b=d} (-1)^a\binom{b}{k-a} G_{a,b}\, c_d.$$
\end{proof}
In \cite[Theorem 1.1]{Intersectionofpermu}, the author found another equivalent (and nicer) formula: $$c_kc_{d-k}\; =\; \sum_{0\leq 2j \leq \min(k, d-k)}(\frac{1}{12})^j \binom{k-j}{j}\binom{d-k-j}{j}.$$
\section{Supplement for Theorem~\ref{thm:main-theorem}}\label{sec:tedious}

In this appendix, we fill in the missing details for the proof of Theorem~\ref{thm:main-theorem}. Specifically, we establish the following numeric lemma claimed in the proof.
\begin{lemma} \label{lem:ugly}
For $t \geq 100$ and $0 < c < 1$, we have the inequality
$$\frac{2cte^{\frac{1}{12t}}}{\sqrt{8ct\pi}}\;\leq\;\left(\frac{2^{2c}(c+1)^{2c+2}}{3e^{2c+1}(c+\frac{1}{2t})^{2c+1}}\right)^{t}.$$
\end{lemma}
\begin{proof}
Let $$f(c) \;=\; \frac{2^{2c}(c+1)^{2c+2}}{3e^{2c+1}(c+\frac{1}{200})^{2c+1}} \; \leq \; \frac{2^{2c}(c+1)^{2c+2}}{3e^{2c+1}(c+\frac{1}{2t})^{2c+1}}.$$

We have $$f'(c) \,=\, \frac{(2^{2c}(c+1)^{2c+2})(3e^{2c+1}(c+\frac{1}{200})^{2c+1})}{(3e^{2c+1}(c+\frac{1}{200})^{2c+1})^2}\left(\log 4 + 2 + 2\log(c+1)-2-2\log(c+\frac{1}{200})-\frac{2c+1}{c+\frac{1}{200}}\right).$$
We want to prove $f'(c) < 0$ for $0 < c < 1$, which is equivalent to $$\log 4 \;+\; 2\log(c+1)\;-\;2\log(c+\frac{1}{200})\;-\;\frac{2c+1}{c+\frac{1}{200}} \;\leq \;0,$$ and can be verified.

Therefore, $$\left(\frac{2^{2c}(c+1)^{2c+2}}{3e^{2c+1}(c+\frac{1}{200})^{2c+1}}\right)^{t} \;=\; f(c)^t$$ achieves minimum when $c = 1$, and $f(c) = 1.046\ldots$. Finally, $$\frac{2cte^{\frac{1}{12t}}}{\sqrt{8ct\pi}}\;<\;2\sqrt{t}\;<\;1.046^t\;<\;f(1)^t\;<\;\left(\frac{2^{2c}(c+1)^{2c+2}}{3e^{2c+1}(c+\frac{1}{2t})^{2c+1}}\right)^{t}$$ when $t\geq 100$.
\end{proof}

Finally, we provide the pseudocode used to verify the small cases for \eqref{eq:centralmomentnormal}.
\noindent

\begin{algorithm}[H] \label{alg:verify}
\DontPrintSemicolon
\KwIn{$d_{\max}$, $t_{\max}$}
\For{$d = 0$ \KwTo $d_{\max}$}{
    Precompute Eulerian numbers $A(d+1,j)$ for $0 \le j \le d$\ by the recurrence $A(n, m) = (n-m)A(n-1, m-1)+(m+1)A(n-1, m), \quad A(1,0) = 1$.\;
    
    \For{$t = 1$ \KwTo $t_{\max}$}{
        Compute the $2t$-th central moment $E_{2t}(d)$ of the Eulerian numbers:\;
        $E_{2t}(d) \gets \left(\sum_j (j - d/2)^{2t} \cdot A(d+1,j) \right)\, /\, (d+1)!$\;
        Check $E_{2t}(d) \leq \left(\frac{d+2}{12}\right)^t(2t-1)!!$ \tcp{Calculations can be done with integers.} 
    }
}
\caption{Pseudocode for verifying the moment inequality}
\end{algorithm}